\documentclass{amsart}

\usepackage{amsmath,amsfonts,amssymb,amsthm}
\usepackage[hidelinks]{hyperref}
\usepackage{mathtools}
 \usepackage{pdfpages}
\usepackage[capitalize]{cleveref}
\usepackage{stmaryrd}
\usepackage{tikz-cd}
\usepackage{mathabx}
\usepackage{graphicx}
\usepackage{adjustbox}

\newtheorem{theorem}{Theorem}[section]

\newtheorem{proposition}[theorem]{Proposition}
\newtheorem{lemma}[theorem]{Lemma}
\newtheorem{corollary}[theorem]{Corollary}

\newtheorem{conjecture}[theorem]{Conjecture}
\newtheorem{problem}[theorem]{Problem}

\theoremstyle{definition}
\newtheorem{definition}[theorem]{Definition}
\newtheorem{example}[theorem]{Example}
\newtheorem{remark}[theorem]{Remark}

\DeclareMathOperator{\Ann}{Ann}

\DeclareMathOperator{\supp}{supp}

\DeclareMathOperator{\deck}{Deck}

\DeclareMathOperator{\ab}{ab}
\DeclareMathOperator{\Ima}{im}

\newcommand{\Z}{\mathbb{Z}}
\newcommand{\N}{\mathbb{N}}
\newcommand{\Q}{\mathbb{Q}}

\DeclareMathOperator{\fp}{FP}

\def\injects{\hookrightarrow}
\newcommand{\isom}{\cong}
\newcommand{\normal}[1]{\langle\!\langle #1 \rangle\!\rangle}

\begin{document}

\title{Lifting relations in right orderable groups}

\author{Marco Linton}
\address{Instituto de Ciencias Matemáticas, Fuencarral-El Pardo, 28049, Madrid, Spain}
\email{marco.linton@icmat.es}

\maketitle

\begin{abstract}
In this article we study the following problem: given a chain complex $A_*$ of free $\Z G$-modules, when is $A_*$ isomorphic to the cellular chain complex of some simply connected $G$-CW-complex? Such a chain complex is called realisable. Wall studied this problem in the 60's and reduced it to a problem involving only the second differential $d_2$, now known as the relation lifting problem. We show that if $G$ is right orderable and $d_2$ is given by a matrix of a certain form, then $A_*$ is realisable. As a special case, we solve the relation lifting problem for right orderable groups with cyclic relation module.
\end{abstract}

\section{Introduction}

If $X$ is a CW-complex with $\pi_1(X) = G$, then the cellular chain complex $C_*(\widetilde{X})$, where $\widetilde{X}$ denotes the universal cover of $X$, has the structure of a complex of free $\Z G$-modules with a natural free $\Z G$-bases for $C_n(\widetilde{X})$ in correspondence with $n$-cells of $X$. A natural question that arose in C. T. C. Wall's work on finiteness conditions for CW-complexes \cite{Wa65,Wa66} was the following: given a complex of free $\Z G$-modules $A_*$, together with chosen free bases, does there exist a simply connected $G$-CW-complex $\widetilde{X}$ and a $\Z G$-isomorphism $C_*(\widetilde{X}) \to A_*$ sending natural basis elements of each $C_n(\widetilde{X})$ to the given basis elements of $A_n$?

Call a $\Z G$-complex $A_*$ \emph{realisable} if it admits a positive answer to the above question and call the isomorphism $C_*(\widetilde{X})\to A_*$ a \emph{realisation}. Call a $\Z G$-complex $A_*$ \emph{admissible} if the subcomplex 
\[
\begin{tikzcd}
A_2 \arrow[r, "d_2"] & A_1 \arrow[r, "d_1"] & A_0
\end{tikzcd}
\]
is realisable. Wall proved the following theorem in \cite{Wa66}, reducing the realisability problem to realisability in dimension two.

\begin{theorem}[Wall]
\label{wall_theorem}
Let $X$ be a connected CW-complex with $G = \pi_1(X)$. Let $A_*$ be an admissible complex of free $\Z G$-modules with chosen bases and let $f\colon A_*\to C_*(\widetilde{X})$ be a $\Z G$-chain homotopy equivalence sending basis elements in dimension zero to 0-cells. Then there is a CW-complex $Y$, a cellular homotopy equivalence $g\colon Y\to X$ and a realisation $h\colon C_*(\widetilde{Y})\to A_*$ such that $f\circ h$ is induced by $g$.
\end{theorem}

There are two obvious necessary conditions needed for $A_*$ to be admissible. The first is that $H_0(A_*) = \Z$ and $H_1(A_*) = 0$. The second is that for each $E\in B_1$, we must have $d_1(E) = g_1V_1 - g_0V_0$ for some $g_0, g_1\in G$ and $V_0, V_1\in B_0$. If $A_*$ satisfies these two conditions, then $A_*$ is said to be an \emph{algebraic} complex. The first condition is necessary since $H_0(X) = \Z$ and $H_1(X) = 0$ for any simply connected CW-complex $X$ whereas the second is necessary because of the particularly rigid nature of boundary maps of 1-cells.

Motivated by \cref{wall_theorem}, Wall conjectured that every algebraic $\Z G$-complex $A_*$ is admissible, indicating that the expected answer should be negative. This conjecture became known as the \emph{relation lifting problem} and has been extensively studied over the last few decades. See work of Harlander \cite{Ha00,Ha15,Ha18} for surveys on this and related problems. In the next section, we give the more common equivalent statement in terms of relation modules.

The first counterexample to Wall's conjecture was provided by Dunwoody \cite{Du72}. Dunwoody's counterexample was the one-relator group $G = \Z*\Z/5\Z$. His proof used the Magnus property for free groups and relied on the existence of non-trivial units in the group ring $\Z G$ arising from torsion. Further counterexamples were provided by Bestvina--Brady in \cite{BB97} who constructed examples of groups of type $\fp_2(\Z)$ that are not finitely presented. Indeed, if $G$ is a group of type $\fp_2(\Z)$ that is not finitely presented, then there is an algebraic $\Z G$-complex $A_*$ with $A_i$ finitely generated for all $i\leqslant 2$ (by definition), but no simply connected cocompact $G$-CW-complex can exist.

When the second differential $d_2$ is given by a matrix of a certain form, we establish Wall's conjecture within the class of right orderable groups. A finite (not necessarily square) matrix is \emph{lower trapezoidal} if all entries on the right hand diagonal are non-zero and all entries above it are zero. For a more formal definition encompassing infinite matrices, see \cref{trapezoidal_defn}.

\begin{theorem}
\label{main_intro}
Let $G$ be a right orderable group and let $A_*$ be an algebraic $\Z G$-complex with given free bases. If $d_2$ is given by a lower trapezoidal matrix over the given bases (acting from the right), then $A_*$ is admissible.
\end{theorem}

\cref{main_intro} is a consequence of our main technical result, \cref{R-relative_main}, which establishes a relative version of admissibility and is valid for algebraic $RG$-complexes where $R$ is an arbitrary domain, not necessarily $=\Z$. Note that we do not assume finite generation of the $A_i$. By additionally assuming that $A_i$ is finitely generated for all $i\leqslant 2$, \cref{main_intro} implies that $G$ is finitely presented.

The CW-complexes with cellular chain complexes as those arising in \cref{main_intro} are homotopy equivalent to \emph{reducible} complexes, introduced by Howie in \cite{Ho82}. A $2$-complex $X$ is reducible (without proper powers) if there are subcomplexes $X_0\subset X_1\subset\ldots \subset X$ such that $X_0$ is a 1-complex, $X = \bigcup X_i$ and if each $X_{i+1}$ is obtained from $X_i$ by attaching a 1-cell $e_i$ and a $2$-cell whose attaching map is not freely homotopic into $X_i$ within $X_i\cup e_i$ (and is not a proper power). Howie proved that if $X$ is a reducible complex without proper powers, then $G = \pi_1(X)$ is locally indicable, and hence right orderable by the Burns--Hale Theorem \cite{BH72}. Using Howie's results we show in \cref{lower_trapezoidal} that the boundary map $\partial_2\colon C_2(\widetilde{X})\to C_1(\widetilde{X})$ is given by a lower trapezoidal $\Z G$-matrix. This result was already known to experts, but an explicit reference could not be found. By understanding further the CW-complex that arises from \cref{main_intro}, we prove the following which can be considered as an algebraic converse.

\begin{theorem}
\label{reducible_intro}
Let $X$ be a connected CW-complex and suppose that $G = \pi_1(X)$ is right orderable. The following are equivalent:
\begin{enumerate}
\item $X$ is homotopy equivalent to a (finite) reducible $2$-complex.
\item There is an algebraic $\Z G$-complex
\[
\begin{tikzcd}
A_2 \arrow[r, "d_2"] & A_1 \arrow[r, "d_1"] & A_0
\end{tikzcd}
\]
with $d_2$ given by a lower trapezoidal matrix over the given bases (and with each $A_i$ finitely generated) and a chain homotopy equivalence $A_*\to C_*(\widetilde{X})$ sending basis elements of $A_0$ to natural basis elements of $C_0(\widetilde{X})$.
\end{enumerate}
\end{theorem}

Since any non-zero $1\times n$ matrix is lower trapezoidal, we obtain the following immediate corollary.

\begin{corollary}
\label{1-relator_cor}
If $G$ is a right orderable group and $A_*$ is an algebraic $\Z G$-complex with $A_2 \isom \Z G$, then $A_*$ is admissible. In particular, if $X$ is a connected CW-complex with $G = \pi_1(X)$, then the following are equivalent:
\begin{enumerate}
\item $X$ is homotopy equivalent to a connected $2$-complex with a single $2$-cell.
\item There is an algebraic $\Z G$-complex $\Z G\to A_1\to A_0$ and a chain homotopy equivalence $A_*\to C_*(\widetilde{X})$ sending basis elements of $A_0$ to natural basis elements of $C_0(\widetilde{X})$.
\end{enumerate}
\end{corollary}

\subsection{The relation lifting and relation gap problems}

Wall's conjecture is more commonly known as the \emph{relation lifting problem} and can be equivalently stated in terms of relation modules. Recall that if $F$ is a free group, $N\triangleleft F$ a normal subgroup, then the \emph{relation module} associated with the presentation $G = F/N$ is the left (or right) $\Z G$-module $N_{\ab} = N/[N, N]$. Note that $\Z F$ acts naturally on $N_{\ab}$ via conjugation by $F$ and since $N$ acts trivially, the action descends to an action of $\Z G$. If $A_*$ is an algebraic $\Z G$-complex, then it is a well known fact that $d_2(A_2)$ is isomorphic to a relation module $N_{\ab}$ (see \cref{submodule}) and hence any free basis for $A_2$ corresponds to a choice of generating set for $N_{\ab}$. We may now rephrase Wall's conjecture:

\begin{problem}[Relation lifting problem]
\label{relation_lifting_prob}
Let $F$ be a free group, $N\triangleleft F$ a normal subgroup and let $G = F/N$. If $r_1, \ldots, r_n\subset N_{\ab} = N/[N, N]$ is a set of $\Z G$-module generators for the relation module $N_{\ab}$, does there exist a collection of elements $w_1, \ldots, w_n\in N$ such that $w_i[N, N] = r_i$ for each $i = 1, \ldots, n$ and $\normal{w_1, \ldots, w_n} = N$?
\end{problem}

The first key step in the proof of our main results is a criterion for when a relation module generator can be lifted as in \cref{relation_lifting_prob}. We state below the criterion applied to the case in which the relation module is cyclic. This is a special case of \cref{lift_criterion}.

\begin{theorem}
\label{Weinbaum_converse}
Let $G = F/N$ be a group such that $\Z G$ is a domain and such that $N_{\ab} = (r)$ for some $r\in N_{\ab}$. Let $w\in N$ be a cyclically reduced element such that $r = w[N, N]$. Then $N = \normal{w}$ if and only if no proper non-empty subword of $w$ lies in $N$.
\end{theorem}

A theorem of Weinbaum's \cite{We72} states that if $F$ is a free group and $w\in F$ is a cyclically reduced element, then no proper non-empty subword of $w$ lies in $\normal{w}$. Rather, every proper non-empty subword of $w$ is non-trivial in the quotient one-relator group $G = F/\normal{w}$. Howie then generalised this result to one-relator products of locally indicable groups in \cite{Ho82}. \cref{Weinbaum_converse} can thus be considered as a converse to Weinbaum's theorem.

The second crux of the article involves showing that, under some extra technical hypotheses (which are always satisfied by right orderable groups), one can always find elements $w\in N$ that map to some relation module generator and such that no proper non-empty subword of $w$ lies in $N$. In the general setting, this leads to \cref{main_intro}. We leave the translation of \cref{main_intro} to a result about relation modules to the interested reader and instead only state the version for cyclic relation modules. The following is \cref{cyclic_module}.

\begin{corollary}
\label{1-relator_cor2}
If $G = F/N$ is right orderable and $N_{\ab} = (r)$ for some $r\in N_{\ab}$, then there is an element $w\in N$ such that
\[
w[N, N] = r, \quad \text{and} \quad \normal{w} = N.
\]
In particular, $G$ is a one-relator group.
\end{corollary}

If $d_{\Z G}(N_{\ab})$ denotes the smallest cardinality of a $\Z G$-generating set for $N_{\ab}$ and $d_{F}(N)$ denotes the smallest cardinality of a normal generating set for $N$, then it is immediate that
\[
d_{\Z G}(N_{\ab})\leqslant d_F(N).
\]
By the Bestvina--Brady examples, it is possible that $d_{\Z G}(N_{\ab})<d_F(N) = \infty$. When we further assume that $d_F(N)<\infty$, whether the above can be taken to be equality is known as the \emph{relation gap problem} and remains open.

\begin{problem}[Relation gap problem]
Let $F$ be a free group, $N\triangleleft F$ a normal subgroup and let $G = F/N$. Is
\[
d_{\Z G}(N_{\ab}) = d_F(N)
\]
when $d_F(N)<\infty$?
\end{problem}

Potential counterexamples were proposed by Gruenberg--Linnell who proved in \cite{GL08} that the group $(\Z\times\Z/n\Z)*(\Z\times\Z/m\Z)$ has relation module generated by three elements (when $\gcd(m, n) = 1$), one fewer elements than the number of obvious normal generators for a presentation. Explicit generators of the relation module were provided by Mannan \cite{Ma16}. Further potential counterexamples were proposed by Bridson--Tweedale in \cite{BT07}, who also showed (based on work of Dyer) that if their groups were proven to be counterexamples to the relation gap problem, then they would also obtain counterexamples to Wall's $D(2)$ conjecture. Finally, we also mention that potential counterexamples to the relation gap problem and the $D(2)$ conjecture also arose in the work of Nicholson \cite[Section 8]{Ni21}.

\cref{1-relator_cor2} shows that there can be no counterexample amongst right orderable groups with cyclic relation module, answering \cite[OR5]{BMS02} for right orderable groups. The only previously known result in this direction is due to Harlander \cite{Ha93} who proved that a solvable group with cyclic relation module is a one-relator group.

Our results actually establish something slightly stronger. If $G = F/N$ and $R$ is a ring, the $R$-relation module is the $RG$-module $R\otimes_{\Z}N_{\ab}$. One can pose the relation lifting and relation gap problems also for the $R$-relation module in the same way. A result of Gruenberg \cite{Gr76} implies that for finite groups, it is easy to find examples with relation gap over a field (see also \cite{Ha18}). We give an explicit example of a group $G$ with a reducible presentation with cyclic relation module over $\Q$, but that is not a one-relator group in \cref{example}. Nevertheless, amongst right orderable groups we may still show there is no relation gap over an arbitrary domain.

\begin{corollary}
\label{1-relator_cor3}
If $G = F/N$ is right orderable, $R$ is a domain and the $R$-relation module $R\otimes_{\Z}N_{\ab}$ is generated by a single element $1\otimes r\in R\otimes_{\Z}N_{\ab}$ as an $RG$-module, then there is an element $w\in N$ and a positive integer $k\in \N$ such that
\[
1\otimes w^k[N, N] = 1\otimes r, \quad \text{and} \quad \normal{w} = N.
\]
In particular, $G$ is a one-relator group.
\end{corollary}

\subsection*{Acknowledgements}

We would like to thank John Nicholson for helpful conversations on the $D(2)$ conjecture and comments on this article. This work has received funding from the European Research Council (ERC) under the European Union's Horizon 2020 research and innovation programme (Grant agreement No. 850930) and from the Spanish Ministry of Science, Innovation, and Universities, through the Severo Ochoa Programme for Centres of Excellence in R\&D (CEX2023-001347-S)

\section{Preliminaries}

\subsection{Group rings and engulfing elements}

Rings will always be assumed to have unit $1\neq 0$ and modules will always be assumed to be left modules. In particular, this means that in this article, conjugation will be considered as a left action. The group of units of a ring $R$ will be denoted by $R^{\times}$. The \emph{annihilator} of a subset $S\subset M$ of a (left) $R$-module $M$ is denoted by $\Ann_R(S) = \{r\in R \mid r\cdot s= 0, \forall s\in S\}$.

Let $R$ be a ring and let $G$ be a group. Denote by $RG$, or $R[G]$, the group ring of $G$ with coefficients in $R$. Each element $r\in RG$ is a formal sum
\[
r = \sum_{g\in G} r_g g
\]
where $r_g\in R$ and for all but finitely many elements $g\in G$, we have $r_g = 0$. The support of an element $r\in RG$ is defined as the set
\[
\supp(r) = \{g\in G \mid r_g \neq 0\}\subset G.
\]
Note that $\supp(rs)\subset \supp(r)\supp(s)$ for all $r, s\in RG$.

Now let $F$ be a free left (right) $RG$-module with a given free basis $B\subset F$. For each $b\in B$, denote by $p_b\colon F \to RG$ the corresponding projection map. We extend the definition of support above to free $RG$-modules by defining for all $m\in F$:
\[
\supp_B(m) = \bigcup_{b\in B}\supp(p_{b}(m))\cdot b\subset G\cdot B.
\]

We introduce the following property as it is precisely what we need to prove our main results in \cref{sec:admissibility}.

\begin{definition}
If $F$ is a free left (right) $RG$-module with a free basis $B\subset F$, say a non-zero element $m\in F$ is \emph{left (right) engulfing (with respect to $B$)} if there is an $r\in RG - R$ such that $\supp_B(r\cdot m)\subset \supp_B(m)$, say it is \emph{not left (right) engulfing} otherwise.
\end{definition}

When $F = RG$, left (right) engulfing elements are always assumed to be with respect to the basis $B = \{1\}$.

\begin{remark}
\label{domain_remark}
If the group ring $RG$ does not contain any left (or right) engulfing elements, then it is immediate that $RG$ is a domain. It is also not hard to see that under the same hypothesis, the group of units of $RG$ is $(RG)^{\times} = \{r\cdot g \mid r\in R^{\times}, g\in G\}$.
\end{remark}

\subsection{Group rings without engulfing elements}

A group $G$ is said to have \emph{$k$-unique products} if for any non-empty finite subsets $A, B\subset G$ with $|AB|\geqslant k$, there are elements $a_1b_1, \ldots, a_kb_k\in AB$ uniquely represented as products in $AB$. It is known that 1-unique products, known as \emph{unique products}, implies $2$-unique products by work of Strojnowski \cite{St80}.

Say $G$ has \emph{left $k$-unique products} if for any pair of finite subsets $A, B\subset G$ with $|A|\geqslant k$ and $B$ non-empty, there exist elements $a_1b_1, \ldots, a_kb_k\in AB$ uniquely represented as products in $AB$ with $a_i\neq a_j$ for all $i\neq j$. Define right $k$-unique products similarly. Our motivation for introducing this property is the following lemma.

\begin{lemma}
\label{2-unique_prod}
If $G$ is a left $2$-unique product group and $R$ is a domain, then $RG$ does not contain any left engulfing elements. In particular, every free left $RG$-module does not contain any left engulfing elements with respect to any basis.
\end{lemma}

\begin{proof}
Let $a, b\in RG$ be non-zero elements. Suppose that $\supp(ab)\subset \supp(b)$. If $1\notin \supp(a)$, then $\supp((1+a)b)\subset \supp(b)$ and so, by possibly replacing $a$ with $1+a$, we may assume that $1\in \supp(a)$. Since $G$ has left $2$-unique products, if $\supp(a)\neq \{1\}$, then $|\supp(a)|\geqslant 2$ and so there are two uniquely represented elements $a_1b_1, a_2b_2\in \supp(a)\supp(b)$ with $a_1\neq a_2$. Since $R$ is a domain, $r_{a_1}r_{b_1}\neq 0$ and $r_{a_2}r_{b_2}\neq 0$ and so $a_1b_1, a_2b_2\in \supp(ab)\subset\supp(b)$. Since $1\in \supp(a)$ and $a_1b_1, a_2b_2\in \supp(b)$, this implies that $a_1 = a_2 = 1$, a contradiction. Hence $\supp(a) = \{1\}$ and so $a\in R$. Since $a$ and $b$ were arbitrary, $RG$ does not contain any left engulfing elements.
\end{proof}

A group $G$ is \emph{right (left) orderable} if there is a total order on $G$ such that for any $a, b\in G$ with $a<b$, we have $ag<bg$ ($ga<gb$) for all $g\in G$. The following is \cite[Lemma 13.1.7]{Pa77} combined with the fact that right orderable groups are also left orderable.

\begin{lemma}
\label{ro}
If $G$ is right orderable, then $G$ has left and right $2$-unique products. Hence, if $R$ is a domain, then $RG$ does not contain any left (or right) engulfing elements and nor does any free left (or right) $RG$-module with respect to any basis.
\end{lemma}

We do not know whether there exists a group with unique products, but without left $2$-unique products or a group with left $2$-unique products, but that is not right orderable. We also do not know whether there is a relation between the unique product property and the absence of left (or right) engulfing elements.

\subsection{Relation modules}

If $G$ is a group, denote by $G_{\ab} = G/[G, G]$ the abelianisation of $G$. If $F$ is a free group and $N\triangleleft F$ is a normal subgroup, we say $F/N$ is a \emph{presentation} for $G = F/N$. The \emph{relation module} for the presentation is the left $\Z G$-module $N_{\ab}$ where the $G$ action is induced by the conjugation by $F$. If $R$ is a ring, the \emph{$R$-relation module} for the presentation is the left $RG$-module 
\[
N_R = R\otimes_{\Z}N_{\ab}.
\]
Recall that $G = F/N$ has \emph{type $\fp_2(R)$} if $N_R$ is finitely generated as an $RG$-module. The following appears as \cite[Lemma 2.1]{BS78}.

\begin{lemma}
\label{generating_lemma}
Let $F$ be a free group, $N, P\triangleleft F$ normal subgroups with $N\leqslant P$ and let $R$ be a ring. If $W\subset N$ is a normal generating set for $N$, then the following are equivalent:
\begin{enumerate}
\item $H_1(P/N, R) = 0$.
\item $W$ is a generating set for $P_R$ as a left $RH$-module, where $H = F/P$.
\end{enumerate}
\end{lemma}

If $S\subset F$ is a free basis for $F$, the relation module for $G = F/N$ is best seen as a submodule of $\bigoplus_{s\in S}RG$. In order to describe this submodule, we will need to introduce \emph{Fox derivatives}. Let $F$ be a free group with free basis $S$. The Fox derivatives of $F$ with respect to $S$ are $\Z F$-linear maps
\[
\frac{\partial}{\partial s}\colon \Z F \to \Z F \quad \text{ for $s\in S$}.
\]
satifying $\frac{\partial}{\partial s}(s) = 1$ for all $s\in S$, $\frac{\partial}{\partial s}(t) = 0$ for all $s\neq t\in S$ and
\[
\frac{\partial}{\partial s}(vw) = \frac{\partial}{\partial s}(v) + v\cdot \frac{\partial}{\partial s}(w)
\]
for all $v, w\in F$. These axioms uniquely determine the Fox derivatives. One can derive from the axioms the fact that
\[
\frac{\partial}{\partial s}\left(s^{-1}\right) = -s^{-1}.
\]
We may extend the the Fox derivatives to $RF$ by applying $R\otimes_{\Z}-$.

See \cite[Proposition 5.4]{Br82} for the proof of the following.

\begin{lemma}
\label{relation_module_submodule}
Let $R$ be a ring, let $F$ be a free group, $S\subset F$ a free basis for $F$, let $N\triangleleft F$ and let $W\subset N$ be a set of elements whose images in $N_R$ generate $N_R$ as an $RG$-module. Denoting by $\phi\colon F \to F/N = G$, we have the following exact sequence of $RG$-modules
\[
\begin{tikzcd}
\bigoplus_{w\in W}RG \arrow[r, "d_2"]& \bigoplus_{s\in S}RG \arrow[r, "d_1"] & RG \arrow[r, "d_0"] & R \arrow[r] & 0
\end{tikzcd}
\]
where
\begin{align*}
d_2(1_w) &= \sum_{s\in S}\phi\left(\frac{\partial}{\partial s}(w)\right)\cdot 1_s \quad \text{for $w\in W$}\\
d_1(1_s) &= s - 1 \quad \text{ for $s\in S$}\\
d_0(1) &= 1.
\end{align*}
Furthermore, $\Ima(d_2) \isom N_R$.
\end{lemma}

The $|W|$-by-$|S|$ matrix
\[
J = \left(\phi\left(\frac{\partial}{\partial s}(w)\right)\right)_{w\in W, s\in S}
\]
from \cref{relation_module_submodule} is called the \emph{Jacobi matrix} of the presentation.

\subsection{Cellular chain complexes and relation modules}

If $X$ is a CW-complex, recall that we denote by $C_*(X)$ the associated cellular chain complex with boundary maps $\partial_n\colon C_n(X)\to C_{n-1}(X)$. If $R$ is a ring, denote by $C_*(X, R) = R\otimes_{\Z}C_*(X)$ the chain complex with boundary maps $1\otimes\partial_n\colon C_n(X, R)\to C_{n-1}(X, R)$.  If $c\subset X$ is an $n$-cell, we denote by $[c]\in C_n(X)$ (and $1\otimes[c]\in C_n(X, R)$) the corresponding $n$-chain. The $n$-cycles and $n$-boundaries are the submodules $\ker(1\otimes\partial_n) = Z_n(X, R)\leqslant C_n(X, R)$ and $\Ima(1\otimes\partial_{n+1}) = B_n(X, R)\leqslant C_n(X, R)$ and $H_n(X, R) = Z_n(X, R)/B_n(X, R)$. Here $C_n(X, R)$ is considered as a (left) $R$-module.

Now let $G$ be a group and suppose that $X$ is a $G$-CW-complex. That is, $G$ acts on $X$ by homeomorphisms, preserving the cell structure. We also require that any element of $G$ that fixes a cell setwise also fixes it pointwise. This action naturally turns each $C_n(X, R)$ into a left $RG$-module and each map $1\otimes\partial_n$ into an $RG$-module map. By choosing $G$-orbit representatives of $n$-cells $c\subset X$, we obtain a natural free $RG$-basis for $C_n(X, R)$ consisting of the elements $1\otimes[c]\in C_n(X, R)$. A free $RG$-basis for $C_n(X, R)$ is \emph{natural} if each basis element is of the form $1\otimes[c]$ for some $n$-cell $c\subset X$.

Since our results only involve $2$-dimensional CW-complexes, we shall usually abbreviate $2$-dimensional CW-complex to \emph{$2$-complex} and a $2$-dimensional $G$-CW-complex to \emph{$G$-$2$-complex}.

Let $X$ be a connected $2$-complex, let $T\subset X^{(1)}$ be a spanning tree and denote by $Z$ the CW-complex obtained from $X$ by contracting $T$ to a single 0-cell. This deformation retraction induces a commutative diagram
\[
\begin{tikzcd}
{C_2(\widetilde{Z}, R)} \arrow[r] \arrow[d, equal] & {C_1(\widetilde{Z}, R)} \arrow[r] \arrow[d, hook, bend right]      & {C_0(\widetilde{Z}, R)} \arrow[r] \arrow[d, hook, bend right]      & R \arrow[r] \arrow[d, equal] & 0 \\
{C_2(\widetilde{X}, R)} \arrow[r]           & {C_1(\widetilde{X}, R)} \arrow[r] \arrow[u, two heads, bend right] & {C_0(\widetilde{X}, R)} \arrow[r] \arrow[u, two heads, bend right] & R \arrow[r]           & 0
\end{tikzcd}
\]
where the vertical map $C_2(\widetilde{Z}, R)\to C_2(\widetilde{X}, R)$ is an isomorphism. Since $Z$ has a single $0$-cell, it is a presentation complex for $\pi_1(X)$ with generators corresponding to each 1-cell and relators corresponding to the attaching maps of each $2$-cell. In particular, given choices of natural bases for $C_i(\widetilde{Z}, R)$, the boundary map $C_2(\widetilde{Z}, R) \to C_1(\widetilde{Z}, R)$ can be given in terms of a Jacobi matrix of this presentation.

These observations imply the following lemma.

\begin{lemma}
\label{submodule}
Let $R$ be a ring and let $X$ be a $2$-complex. Then $\Ima(1\otimes\partial_2) \isom N_R$, where $N = \ker\left(\pi_1\left(X^{(1)}\right)\to \pi_1(X)\right)$ and $N_R$ is the $R$-relation module for the presentation $\pi_1(X) \isom \pi_1\left(X^{(1)}\right)/N$.
\end{lemma}

\subsection{A simple lemma on integer sequences}

The following seemingly uninteresting lemma will be important in \cref{sec:hierarchies} when constructing hierarchies for one-relator products.

\begin{lemma}
\label{number_lemma}
Let $a\leqslant b\in \N$ be a pair of coprime integers  and let 
\[
0 = i_1, j_1, i_2, \ldots, j_n, i_{n+1} = 0
\]
be a sequence of non-negative integers with $i_l - j_l \in b\Z$ and $j_l - i_{l+1}\in a\Z$ for each $l$. Either $\sum_{i=1}^n (i_l - j_l) = 0$ or for some $l$ we have $i_l\geqslant a+b-1$ or $j_l\geqslant a+b-1$.
\end{lemma}

\begin{proof}
We assume that $\sum_{i=1}^n(i_l-j_l) \neq 0$ and $i_l, j_l<a+b-1$ for all $l$ and derive a contradiction. The proof proceeds by induction on $n$. Suppose there is some $1\neq p< q$ such that $i_p = i_q$. After possibly subtracting the minimum value from each term and cyclically permuting the subsequence $i_p, j_p, \ldots, i_{q-1}, j_{q-1}, i_q = i_p$, we may ensure that it is of the form required by the lemma and hence conclude by induction that either $\sum_{l = p}^q(i_l - j_l)= 0$ or $i_l\geqslant a+b-1$ or $j_l\geqslant a+b-1$ for some $l$. In the second case we have obtained a contradiction, in the first case we may replace the subsequence $i_p, j_p, \ldots, i_q$ with $i_q$ and also reach a contradiction by induction. Thus, from now on we may assume that $i_p\neq i_q$ for all $1\neq p\neq q$. If $(i_l - j_l) = 0$, then if we replace the subsequence $j_{l-1}, i_l, j_l$ with $j_l+j_{l+1}$, we obtain a shorter sequence also satisfying our assumptions and not changing the value of $\sum_{i=1}^n(i_l - j_l)$. Thus, we obtain a contradiction by induction in this case too. Finally, assume that $i_l - j_l\neq 0$ for all $l$. If there is some $l$ such that $i_l - j_l>0$ and $j_l - i_{l+1}>0$ or $i_l - j_l<0$ and $j_l - i_{l+1}<0$, then $i_{l+1} - i_l \geqslant a+b$ or $i_l - i_{l+1}\geqslant a+b$ and we reach a contradiction. Similarly, if there is some $l$ such that $j_l - i_{l+1}>0$ and $i_{l+1} - j_{l+1}>0$ or $j_l - i_{l+1}<0$ and $i_{l+1} - j_{l+1}<0$, then we reach a contradiction. Since $i_1 - j_1<0$, this implies that $i_l - j_l<0$ and $j_l-i_{l+1}>0$ for all $l$. Moreover, since $a$ and $b$ are coprime and
\begin{align*}
-\sum_{i=1}^n(i_l - j_l) = n_1b = n_2a = \sum_{i=1}^n(j_l - i_{l+1})
\end{align*}
for some $n_1, n_2>0$, we must have $-\sum_{l=1}^n(i_l-j_l) = n_1b\geqslant ab$. If $-(i_l - j_l)\geqslant 2b$ for some $l$, then $j_l> 2b\geqslant a+b$, a contradiction. Thus we have that $(i_l - j_l) = -b$ for all $l$. In particular, we have that $n\geqslant a$. Since $i_p\neq i_q$ for all $1\neq p\neq q$, there are at least $a$ distinct non-zero values of $i_l$ and hence $i_l\geqslant a-1$ for some $l$. Then, combined with the fact that $j_l - i_l = b$, we have $j_l\geqslant a+b-1$, a contradiction.
\end{proof}

\section{One-relator products and elementary reductions}

If $A$ and $B$ are groups, we shall regard elements $w\in A*B - 1$ as words over the alphabet $(A-1)\sqcup (B-1)$. A word $w\in A*B$ then has a well-defined length over this alphabet. It is freely reduced if there is no subword of $w$ consisting only of letters from $A - 1$ or $B-1$ that is equal to the identity in $A$ or $B$ respectively. It is cyclically reduced if each cyclic rotation of $w$ is also freely reduced.

Let $A$ and $B$ be groups and let $w\in A*B$ be a cyclically reduced word. We call $G = \frac{A*B}{\normal{w}}$ a \emph{one-relator product}. If $A$ and $B$ are both locally indicable, Howie proved in \cite{Ho81} that the maps $A\to G$ and $B\to G$ are actually injective, a result which can be considered a generalisation of Magnus' Freiheitssatz for one-relator groups. The following statement combines this result with \cite[Theorem 4.2 \& Corollary 3.4]{Ho82}.

\begin{theorem}[Howie]
\label{inclusion}
Let $A$ and $B$ be locally indicable groups and let $w\in A*B$ be a cyclically reduced word of length at least two that is not a proper power. Let $n\geqslant 1$ be an integer and denote by $G = \frac{A*B}{\normal{w^n}}$. The following hold
\begin{enumerate}
\item The natural maps $A, B\to G$ are injective.
\item If $u$ is a proper subword of (a cyclic conjugate of) $w^n$, then $u \neq_G 1$.
\item If $n = 1$, then $G$ is locally indicable.
\end{enumerate}
\end{theorem}

Let $Z$ be a $2$-complex and $X\subset Z$ a subcomplex. We say that $X\subset Z$ is an \emph{elementary reduction} if $Z - X$ consists of a single 1-cell $e$ and at most one $2$-cell whose attaching map is not freely homotopic within $X\cup e$ into $X$. A $2$-complex $Z$ is \emph{reducible to $X\subset Z$} if there are subcomplexes $X = X_0 \subset X_1\subset \ldots \subset Z$ such that $Z = \bigcup X_i$ and $X_i\subset X_{i+1}$ is an elementary reduction for all $i$. A $2$-complex $X$ is \emph{reducible} if it is reducible to $X^{(0)}$. If $X$ is reducible, say $X$ is \emph{without proper powers} if for each $i$ such that $X_{i+1} - X_i = e\cup c$ with $e$ a 1-cell and $c$ a $2$-cell, the attaching map for $c$ is not equal to a proper power in $\pi_1(X_i\cup e, x)$. We state a corollary to \cref{inclusion} in terms of the language we have just introduced.

\begin{corollary}
\label{inclusion_corollary}
Let $Z$ be a connected $2$-complex and let $X\subset Z$ be an elementary reduction such that $Z - X = e\cup c$, where $e$ is a 1-cell and $c$ is a $2$-cell. If $\pi_1(X, x)$ is locally indicable for every choice of basepoint $x\in X$, then:
\begin{enumerate}
\item The map $\pi_1(X, x)\to \pi_1(Z, x)$ induced by inclusion is injective for each choice of basepoint $x\in X$.
\item If $\tilde{e}, \tilde{c}\subset \widetilde{Z}$ are lifts of $e, c\subset Z$ to the universal cover $\widetilde{Z} \to Z$, then the difference between the number of times the attaching map for $\tilde{c}$ traverses $\tilde{e}$ in one direction and the number of times it traverses $\tilde{e}$ in the other direction is either $-1$, $0$ or $1$.
\item If the attaching map for $c$ is not a proper power in $\pi_1(X\cup e)$, then $\pi_1(Z)$ is locally indicable.
\end{enumerate}
\end{corollary}

\begin{proof}
We may assume that $Z$ is connected. Choose an orientation on $e$ so that we may consider $e$ as a path and let $e^{-1}$ denote the path traversing $e$ in the opposite direction. Let
\[
w = e^{\epsilon_0}*p_0*\ldots*e^{\epsilon_n}*p_n
\]
be the loop in $Z$ determined by the attaching map for $c$, well-defined up to cyclic conjugation and inverting, where each $p_i$ is a path supported in $X$ and where $\epsilon_i\in\{\pm1\}$. Such a factorisation exists because the attaching map of $c$ does not lie in $X$. Now there are two cases to consider, according to whether $X$ is connected or not. 

Let us first suppose that $X$ is connected. Then the two endpoints of $e$ lie in $X$ and we have $\pi_1(X\cup e) \isom \pi_1(X)*\Z$. Since the attaching map of $c$ is not freely homotopic into $X$, we may apply \cref{inclusion} to obtain the first and third statements. Now let $\tilde{w} = \tilde{e}_0^{\epsilon_0}*\tilde{p}_0*\ldots*\tilde{e}_n^{\epsilon_n}*\tilde{p}_n$ be the lift of $w$ corresponding to the attaching map of the lift $\tilde{c}$ of $c$. The proof proceeds by induction on $n$. If $w$, as a word over $\pi_1(X)*\Z$, is cyclically reduced, then we may apply \cref{inclusion} to conclude that there is no proper non-empty subpath of the attaching map of $\tilde{c}$ that traverses $\tilde{e}^{\pm1}$ and that is a loop. Hence, $\tilde{w}$ traverses $\tilde{e}$ at most once in one of the two directions and so the second conclusion holds. This, for instance, holds if $n = 0, 1$. Now assume the inductive hypothesis and suppose that $w$ is not cyclically reduced as a word over $\pi_1(X)*\Z$. Then, for some $i$, we have that $\epsilon_i = -\epsilon_{i+1}$ and $p_i$ is a nullhomotopic loop in $X$. Here we are considering the indices modulo $n+1$. This implies that $\tilde{e}_i^{\epsilon_i}*\tilde{p}_i*\tilde{e}_{i+1}^{\epsilon_{i+1}}$ is a loop and hence that $\tilde{e}_i^{\epsilon_i}$ and $\tilde{e}_{i+1}^{\epsilon_{i+1}}$ together contribute $0$ to the difference between the number of times $\tilde{w}$ traverses $\tilde{e}$ and the number of times it traverses $\tilde{e}^{-1}$. By removing $e^{\epsilon_i}*p_i*e^{\epsilon_{i+1}}$ from $w$, we obtain a new path $w'$ which is strictly shorter than $w$. Since $w$ is freely homotopic to $w'$ in $X\cup e$, we may apply the inductive hypothesis to conclude that the difference between the number of times $\tilde{w}'$ traverses $\tilde{e}$ and $\tilde{e}^{-1}$ is either $-1, 0$ or $1$. Hence, the same holds for $\tilde{w}$.

The case in which $X$ is not connected is almost identical and so we leave it to the reader.
\end{proof}

\begin{corollary}
\label{boundary_corollary}
Let $R$ be a ring, let $Z$ be a connected $2$-complex and let $X\subset Z$ be an elementary reduction such that $Z - X = e\cup c$, where $e$ is a 1-cell and $c$ is a $2$-cell. Suppose that $\pi_1(X, x)$ is locally indicable for every choice of basepoint $x\in X$. Denote by $G = \pi_1(Z)$, let $p\colon \widetilde{Z}\to Z$ be the universal cover, let $\overline{X} = p^{-1}(X)$ and let $\tilde{e}, \tilde{c}\subset \widetilde{Z}$ be lifts of $e, c\subset Z$. Then 
\[
(1\otimes\partial_2)(1\otimes[\tilde{c}])\in C_1(\overline{X}, R) + r\cdot [\tilde{e}]
\]
for some $0\neq r = \sum_{g\in G}r_g\cdot g\in RG$ such that $r_g\in \{-1, 0, 1\}$ for all $g\in G$.
\end{corollary}

We now state the precise version of the definition of a lower trapezoidal matrix given in the introduction.

\begin{definition}
\label{trapezoidal_defn}
Denote by $[n]\subset \N$ the subset $\{1, 2, \ldots, n\}$ if $n\in \N$ and by $[n] = \N$ if $n=\infty$. If $M = \{m_{ij}\}_{i\in [m], j\in [n]}$ is an $m\times n$ matrix, we say that $M$ is \emph{lower trapezoidal} if for each $i\in [m]$, there is a $j_i\in [n]$ such that $m_{ij} = 0$ for all $j>j_i$ and such that $j_i< j_{i+1}$ for all $i, i+1\in [m]$. Call the entries $m_{ij_i}$ the diagonal entries of $M$.
\end{definition}

After possibly reordering the columns, \cref{trapezoidal_defn} coincides with the definition given in the introduction for finite lower trapezoidal matrices.

The following is one direction of our main theorem. This direction was essentially already known, but a satisfactory reference could not be found.

\begin{theorem}
\label{lower_trapezoidal}
Let $X$ be a reducible $2$-complex without proper powers and let $R$ be a ring. Then $1\otimes\partial_2\colon C_2(\widetilde{X}, R) \to C_1(\widetilde{X}, R)$ is given by a lower trapezoidal matrix over any natural basis.
\end{theorem}

\begin{proof}
Let $X_0\subset X_1\subset \ldots \subset X$ be the subcomplexes such that $X_0 = X^{(0)}$, $X_{i}\subset X_{i+1}$ is an elementary reduction for each $i$ and $\bigcup X_i = X$. Since $X$ is without proper powers, we may apply \cref{inclusion} and induction to conclude that $\pi_1(X_i)\to \pi_1(X_{i+1})\to\pi_1(X)$ is injective for each choice of basepoint.

Let $c_i$ be the $i^{\text{th}}$ $2$-cell appearing in $X_1 - X_0, X_2 - X_1, \ldots$ and let $e_{i}$ be the 1-cell in $X_{i} - X_{i-1}$. Then let $j_i$ denote the index such that $c_i\subset X_{j_i} - X_{j_i - 1}$. A natural basis for $C_2(\widetilde{X}, R)$ corresponds to choices of lifts $\tilde{c}_i\subset \widetilde{X}$ of each $c_i$. Similarly, a natural basis for $C_1(\widetilde{X}, R)$ corresponds to choices of lifts $\tilde{e}_i\subset \widetilde{X}$ of each $e_i$. The matrix $M$ defining $1\otimes\partial_2$, with the given order on the basis, has $ij$-entry the projection of $1\otimes\partial_2([c_i])$ to the $1\otimes[e_j]$ factor. The $ij$-entry of $M$ is $0$ when $j>j_i$ since the attaching map for $c_i$ does not traverse $e_j$. Now all we need to show is that the $ij_i$-entry is non-zero for all $i$. Letting $\overline{X}_i = p^{-1}(X_i)$, where $p\colon \widetilde{X}\to X$ is the universal cover, this follows from the fact that
\[
C_i(\overline{X}_i, R) \isom \bigoplus_{X_i^{\alpha}\in \pi_0(X_i)}C_i(\widetilde{X}_i^{\alpha}, R)\otimes_{R[\pi_1(X_i^{\alpha})]}RG
\]
combined with \cref{boundary_corollary}.
\end{proof}

\section{Hierarchies for one-relator products}
\label{sec:hierarchies}

The aim of this section will be to prove the following hierarchy result for one-relator products. A version of this result is implicit in the work of Howie, see \cite{Ho82}.

\begin{theorem}
\label{hierarchy}
Let $A$ and $B$ be locally indicable groups, $w\in A*B$ an element not conjugate into $A$ or $B$ and not a proper power and $n\geqslant 1$ an integer. Then $G\isom \frac{A*B}{\normal{w^n}}$ can be obtained from subgroups of $A$ and $B$, free groups and free products of $\Z/n\Z$ by performing finitely many amalgamated free products and HNN-extensions.
\end{theorem}

The reader is directed to \cite{serre_80} for the relevant background on amalgamated free products, HNN-extensions and, more generally, graphs of groups.

\subsection{HNN-splittings via infinite cyclic covers}

In order to find HNN-splittings, we will need the notion of a $\Z$-domain, introduced in \cite{Lin22}. Recall that a $\Z$-cover is a covering space $\rho\colon Y\to X$ with $\deck(p) \isom \Z$.

\begin{definition}
Let $\rho\colon Y\to X$ be a $\Z$-cover and let $t\in \deck(\rho)$ be a generator. A connected subcomplex $D\subset Y$ is a \emph{$\Z$-domain} for $\rho$ if the following hold:
\begin{enumerate}
\item $\bigcup_{i\in \Z} t^i\cdot D = Y$.
\item $D\cap t\cdot D$ is connected and non-empty.
\item $D\cap t^i\cdot D\subset D\cap t^{i-1}\cdot D$ for all $i\geqslant 1$.
\end{enumerate}
\end{definition}

The point of this definition is that a $\Z$-domain in a $\Z$-cover gives rise to a HNN-splitting of the fundamental group of the base space. The following is \cite[Proposition 4.3]{Lin22} made explicit.

\begin{lemma}
\label{find_splitting}
Let $\rho\colon Y\to X$ be a $\Z$-cover, let $t\in \deck(\rho)$, let $D\subset Y$ be a $\Z$-domain for $p$ and denote by $D_0 = t^{-1}\cdot D\cap D$. If $\pi_1(D_0)\to \pi_1(D)$ and $\pi_1(t\cdot D_0)\to \pi_1(D)$ are injective, then
\[
\pi_1(X) \isom \pi_1(D)*_{\psi}
\]
where $\psi\colon \pi_1(D_0) \to \pi_1(t\cdot D_0)$ is induced by translation by $t$.
\end{lemma}

In order to be able to apply \cref{find_splitting} to our setting, we shall need to be able to find $\Z$-domains.

\begin{proposition}
\label{Z-domain}
Let $X\subset Z$ be an elementary reduction such that $\pi_1(X, x)$ is indicable for each choice of basepoint, such that $Z - X = e\cup c$ and such that the attaching map for $c$ is surjective on the 1-skeleton of $Z$. There is a $\Z$-cover $\rho\colon Y\to Z$, lifts $\tilde{c}, \tilde{e}\subset Y$ of $c, e$ such that the connected component $D$ of
\[
\rho^{-1}(X) \cup \tilde{c} \cup \bigcup_{m\leqslant i\leqslant M} t^i\cdot \tilde{e}
\]
containing $\tilde{c}$ is a $\Z$-domain, where here $m$ and $M$ are, respectively, the minimal and maximal integers such that the attaching map for $\tilde{c}$ traverses $t^m\cdot \tilde{e}$ and $t^M\cdot \tilde{e}$.
\end{proposition}

\begin{proof}
Denote by $e_i = t^i\cdot \tilde{e}$ and by $\lambda$ the attaching map for $c$. By replacing $\tilde{c}$ with $t^{-m}\cdot \tilde{c}$, we may assume that $m = 0$. We have two cases to consider corresponding to whether $X$ is connected or not.

Suppose first that $X$ is not connected, denote by $X_1, X_2$ the two connected components. We have $H^1(Z)\isom (H^1(X_1)\oplus H^1(X_2))/(a, b)$ for some $a\in H^1(X_1)$ and $b\in H^1(X_2)$. Since $\pi_1(X_i)$ is indicable, $H^1(X_i)$ is free abelian of rank at least one for $i = 1, 2$. If $H^1(X_1)$ is not infinite cyclic or $a = 0$, there is an epimorphism $\phi\colon \pi_1(Z)\to \Z$ such that $\phi(\pi_1(X_2)) = 0$. Consider the cyclic cover $\rho\colon Y\to Z$ corresponding to this epimorphism. We have that $\rho^{-1}(X_1)$ is a connected infinite cyclic cover of $X_1$ and $\rho^{-1}(X_2)$ is a disjoint union of copies of $X_2$, each connected to $\rho^{-1}(X_1)$ via a (unique) edge $e_i$. In this case, it is not hard to see that the three conditions for $D$ to be a $\Z$-domain hold. By symmetry, if $H^1(X_2)$ is not cyclic or $b = 0$, there is a cyclic cover in which $D$ is a $\Z$-domain.

So now suppose that $H^1(X_1) \isom H^1(X_2)\isom \Z$ and that $a, b\neq 0$. We may also assume that $a\leqslant b$ and that $a$ and $b$ are coprime. Then there is a unique (up to change of sign) epimorphism $\phi\colon \pi_1(Z)\to \Z$ given by $\phi(\pi_1(X_1)) = b\Z$, $\phi(\pi_1(X_2)) = a\Z$. It follows that $\rho^{-1}(X_1) = \bigsqcup_{i=0}^{b - 1}\overline{X}_{1, i}$ where $\overline{X}_{1, i}\to X_1$ are copies of the induced connected cyclic cover and $\rho^{-1}(X_2) = \bigsqcup_{i=0}^{a - 1}\overline{X}_{2, i}$ where $\overline{X}_{2, i}\to X_2$ are copies of the induced connected cyclic cover. For each $i\in \Z$, the edge $e_i$ connects $\overline{X}_{1, i\mod(b)}$ with $\overline{X}_{2, i\mod(a)}$. Let $w = a_1b_1\ldots a_nb_n$ be the word given by $\lambda$ (unique up to cyclic conjugation and inversion), with $a_i\in \pi_1(X_1)$, $b_i\in \pi_1(X_2)$ for $i = 1, \ldots, n$. Consider the sequence 
\begin{align*}
i_1 &= \phi(\epsilon) = 0& j_1 &= \phi(a_1),\\
i_2 &= \phi(a_1b_1) & j_2 &= \phi(a_1b_1a_2),\\
&\vdots &  &\vdots\\
i_n &= \phi(a_1b_1\ldots b_{n-1})& j_n &= \phi(a_1b_1\ldots a_n),\\
i_{n+1} &= \phi(a_1b_1\ldots a_nb_n) = 0. & &
\end{align*}
We have $i_l - j_l\in b\Z$ and $j_l - i_{l+1}\in a\Z$. The integer $M -m$ is the difference between the maximum and minimum integer in this sequence. After possibly cyclically permuting and inverting $w$, we may assume that $i_l, j_l\geqslant 0$ for each $l$. We have
\begin{equation}
\label{eq}\sum_{l=1}^n(i_l - j_l) = -\sum_{l=1}^n(j_l-i_{l+1}) = kab\neq 0
\end{equation}
for some $k\in \Z$. By \cref{number_lemma} and \cref{eq}, we have $M - m \geqslant a+b -1$. This implies that 
\[
B = \rho^{-1}(X)\cup e_0\cup\ldots\cup e_{a+b - 2} \subset D_0.
\]
Since $\rho^{-1}(X)$ consists of precisely $a+b$ components and each addition of a new edge $e_i$ as $i$ ranges from $0$ to $a+b -2$ decreases the number of components by one, it follows that $B$ is connected. Since $D_0$ and $D$ are obtained from the above by adding 1-cells with boundary 0-cells lying in $\rho^{-1}(X)$ to $B$, it follows that $D_0$ and $D$ are connected.

If $X$ is connected, then $H^1(Z)\isom (H^1(X)\oplus H^1(C))/(a, b)$ for some $a\in H^1(X)$ and $b\in H^1(C)$, where $C\subset Z$ is any embedded cycle containing the edge $e$. Now the argument proceeds in essentially the same way as in the previous case. We leave this as an exercise to the reader.
\end{proof}

With this set-up, we may state the following useful variation of the classical Magnus splitting theorem. A version of this appeared in \cite{Ho82} and the case where $Z$ contains a single $2$-cell was handled in detail in \cite{Lin22}.

\begin{theorem}
\label{HNN_one-relator}
Let $X\subset Z$ be an elementary reduction such that $\pi_1(X, x)$ is locally indicable for each choice of basepoint, such that $Z - X = e\cup c$ and such that the attaching map for $c$ is surjective on the 1-skeleton of $Z$. There is a $\Z$-cover $\rho\colon Y\to Z$, lifts $\tilde{c}, \tilde{e}\subset Y$ of $c, e$ such that the connected component $D$ of
\[
\rho^{-1}(X) \cup \tilde{c} \cup \bigcup_{m\leqslant i\leqslant M} t^i\cdot \tilde{e}
\]
containing $\tilde{c}$ is a $\Z$-domain, where here $m$ and $M$ are, respectively, the minimal and maximal integers such that the attaching map for $\tilde{c}$ traverses $t^m\cdot \tilde{e}$ and $t^M\cdot \tilde{e}$. Moreover, denoting by $D_0 = t^{-1}\cdot D\cap D$, the following hold:
\begin{enumerate}
\item The homomorphisms $\pi_1(D_0), \pi_1(t\cdot D_0)\to \pi_1(D)$ induced by inclusion are injective.
\item If $\psi\colon \pi_1(D_0)\to\pi_1(t\cdot D_0)$ is the isomorphism induced by translation by $t$, then
\[
\pi_1(Z) \isom \pi_1(D)*_{\psi}.
\]
\end{enumerate}
\end{theorem}

\begin{proof}
Let $\rho\colon Y\to Z$ and $D\subset Y$ be the $\Z$-cover and $\Z$-domain from \cref{Z-domain}. Since $D_0$ and $t\cdot D_0$ are homotopic to a wedge of copies of covers of components of $X$ and cycles $S^1$, we have that $\pi_1(D_0)$ and $\pi_1(t\cdot D_0)$ split as free products of free groups and subgroups of $\pi_1(X, x)$ for different choices of basepoint $x$. Thus, since $\pi_1(X, x)$ is locally indicable for all choices of basepoint, it follows that $\pi_1(D_0)$ and $\pi_1(t\cdot D_0)$ are locally indicable also. Thus, since $D_0, t\cdot D_0\subset D$ are elementary reductions to $2$-complexes with locally indicable fundamental groups, the homomorphisms $\pi_1(D_{0}), \pi_1(D_{1})\to \pi_1(D)$ induced by inclusion are injections by \cref{inclusion}. We may now apply \cref{find_splitting} to obtain the desired HNN-splitting.
\end{proof}

\subsection{The steps in the hierarchy}

The following useful lemma is due to Howie, see the proof of \cite[Theorem 4.3]{Ho81}.

\begin{lemma}
\label{amalgam}
Let $A$ and $B$ be locally indicable groups and let $w\in A*B$ be a cyclically reduced word of length at least two. If $G = \frac{A*B}{\normal{w}}$, then 
\[
G \isom A\underset{A_0}{*}H\underset{B_0}{*}B
\]
where $A_0$ and $B_0$ are the subgroups of $A$ and $B$ generated by the elements mentioned by $w$ and $H = \frac{A_0*B_0}{\normal{w}}$.
\end{lemma}

We now translate \cref{HNN_one-relator} to the language of one-relator products.

\begin{theorem}
\label{HNN}
Let $A$ and $B$ be locally indicable groups, let $w \in A*B$ be a cyclically reduced word of length at least two and let $n\geqslant 1$ be an integer. If the letters appearing in $w$ generate $A*B$, then, denoting by $G = \frac{A*B}{\normal{w^n}}$, there is a HNN-splitting $G = H*_{\psi}$ where 
\[
H \isom \frac{J_0*\Z}{\normal{u^n}} \isom \frac{\Z*J_1}{\normal{v^n}}
\]
and where $\psi\colon J_0\to J_1$ identifies $J_0$ with $J_1$. Moreover, the following holds:
\begin{enumerate}
\item $J_0, J_1$ split as free products of subgroups of $A$ and $B$ and a free group.
\item The length of $u$ and $v$ as cyclically reduced words over $J_0*\Z$ and $\Z*J_1$ respectively are strictly less than that of $w$.
\end{enumerate}
\end{theorem}

\begin{proof}
We shall assume that $n = 1$, the argument is unchanged for $n\geqslant 2$. Let $X_A$ be any $2$-complex with $\pi_1(X_A) = A$, $X_B$ be any $2$-complex with $\pi_1(X_B) = B$ and denote by $X = X_A\sqcup X_B$. Then $Z$ is obtained by attaching a 1-cell $e$ connecting $X_A$ to $X_B$ and a $2$-cell whose attaching map spells out the word $w$. Then $\pi_1(Z) = G$ and $X\subset Z$ is an elementary reduction. Since the letters appearing in $w$ generate $A*B$, we may assume that the attaching map $\lambda\colon S^1\to Z^{(1)}$ is surjective. Applying \cref{HNN_one-relator}, we find that there is an infinite cyclic cover $\rho\colon Y\to Z$ and a subcomplex $D\subset Y$ such that $G\isom \pi_1(D)*_{\psi}$, where $\psi\colon \pi_1(D_0)\to\pi_1(t\cdot D_0)$ is the isomorphism induced by translation by a generator of the deck group $\deck(\rho) = \langle t\rangle$, $t^{-1}\cdot D\cap D  = D_0 \to t\cdot D_0 = D\cap t\cdot D$. Moreover, $D$ supports a single lift $\tilde{\lambda}$ of $\lambda$, $D_0$ is connected and $D_0\subset D$, $t\cdot D_0\subset D$ are elementary reductions. Hence, we have
\[
\pi_1(D) = \frac{\pi_1(D_0) * \Z}{\normal{u}} = \frac{\Z * \pi_1(t\cdot D_0)}{\normal{v}}
\]
where $u$ and $v$ are the words spelled out by this single lift of $\lambda$. Since $\pi_1(D_0), \pi_1(t\cdot D_0)$ are subgroups of $A*B$, they each split as free products of subgroups of $A$ and $B$ and a free group by the Kurosh subgroup theorem. 

All that remains to be shown is that the lengths of $u$ and $v$ are strictly smaller than that of $w$. We first claim that $m<M$.  If $m = M$, then this would imply that every letter appearing in $w$ from $A$ and every letter appearing in $w$ from $B$ lies in the kernel of the epimorphism $G\to \Z$ induced by the infinite cyclic cover. However, seeing as these letters actually generate $G$ by assumption, this is a not possible and so $m<M$ as claimed.

The length of $u$ is precisely twice the number of times that $\tilde{\lambda}$ traverses $e_M$, whereas the length of $v$ is precisely twice the number of times that $\tilde{\lambda}$ traverses $e_m$. On the other hand, the length of $w$ is twice the number of times $\lambda$ traverses $e$ and hence twice the sum of the number of times $\tilde{\lambda}$ traverses each $e_i$ for $m\leqslant i\leqslant M$. Hence, since $m<M$, the lengths of $u$ and $v$ are both strictly smaller than that of $w$.
\end{proof}

\begin{remark}
\label{relator_form}
We actually proved something stronger on the form of the word $u$ from \cref{HNN}. If $w = a_1b_1\ldots a_nb_n$ and $H = J_0*\langle t\rangle$, then there are non-trivial elements $c_1, \ldots, c_{k+1}\in A*B$ so that $w = c_1c_2\ldots c_{k+1}$ with $k<n$ and there are words $t_{-1}, t_{+1}\in A*B$ such that 
\begin{align*}
t &= t_{-1}t_{+1}^{-1},\\
x_1 &= c_1t_{-\epsilon_1}\in J_0,\\
x_i &= t_{\epsilon_{i-1}}^{-1}c_it_{-\epsilon_i}\in J_0 \quad  \text{for } 1< i< k,\\
x_{k+1} &= t_{-\epsilon_{k}}^{-1}c_{k+1}\in J_0
\end{align*}
and such that
\[
u = x_1t^{\epsilon_1}\ldots x_kt^{\epsilon_k}x_{k+1}
\]
as a word over $J_0*\langle t\rangle$, where $\epsilon_i\in \{\pm1\}$. As an element of $A*B$, $u$ is equal to $w$. Using the notation from the proof of \cref{HNN}, the elements $t_{-1}, t_{+1}\in A*B$ correspond to paths in $D_0$ connecting the basepoint with the origin of $e_M$ and the target of $e_M$ respectively. Then each $c_i$ corresponds to a maximal subword of $w$ whose associated path is supported in $D_0$. Since they are maximal, they correspond to subpaths of the path corresponding to $w$ such that immediately before and immediately after, the edge $e_M$ is traversed.

The same can be said about $v$.
\end{remark}

Combining \cref{amalgam} and \cref{HNN} and induction, we obtain \cref{hierarchy}.

\section{Normal $R$-perfect subgroups of one-relator products}

The aim of this section is to understand perfect normal subgroups of one-relator products. To this end, we prove the following theorem.

\begin{theorem}
\label{perfect_subgroup}
Let $R$ be a ring, $k\geqslant 1$ an integer, $A$ and $B$ locally indicable groups, let $w\in A*B$ be a cyclically reduced word of length at least two and consider the one-relator product $G = \frac{A*B}{\normal{w^k}}$. Let $K\trianglelefteq G$ be a normal subgroup such that $K\cap A = 1 = K\cap B$ and, if $\phi\colon G\to G/K = Q$ denotes the quotient map, such that $RQ$ is a domain and $\phi(u)\neq 1$ for all proper non-empty subwords $u$ of $w$. If $H_1(K, R) = 0$, then $k$ is a unit in $R$ and $K = \normal{w}$.
\end{theorem}

Using \cref{perfect_subgroup}, we may deduce a criterion for when an element is a lift of a generator of a (relative) cyclic relation module.

\begin{theorem}
\label{lift_criterion}
Let $R$ be a ring, let $A$ and $B$ be locally indicable groups and let $N\triangleleft A*B$ be a normal subgroup such that $N\cap A = 1 = N\cap B$ and such that $RQ$ is a domain, where $Q = \frac{A*B}{N}$. Suppose that $N_{R} = (r)$ for some $r\in N_{R}$ and let $w\in N$ be a cyclically reduced element such that $1\otimes w[N,N] = r$. Then $N = \normal{w}$ if and only if no proper non-empty subword of $w$ lies in $N$.
\end{theorem}

\begin{proof}
If $N = \normal{w}$, then no proper non-empty subword of $w$ lies in $N$ by \cref{inclusion}. If no proper non-empty subword of $w$ lies in $N$, then we may combine \cref{generating_lemma} with \cref{perfect_subgroup} to deduce that $N = \normal{w}$.
\end{proof}

Before moving onto the proof of \cref{perfect_subgroup}, we need to prove a couple of useful lemmas.

\begin{lemma}
\label{homology_injection}
Let $R$ be a ring, $k\geqslant 1$ an integer, $A$ and $B$ groups, let $w = a_1b_1\ldots a_nb_n\in A*B$ be a cyclically reduced word and consider the one-relator product $G = \frac{A*B}{\normal{w^k}}$. Let $K\trianglelefteq G$ be a normal subgroup such that, if $\phi\colon G\to G/K = Q$ denotes the quotient map, then
\[
\Ann_{RQ}\left(\left(\sum_{i=0}^{k-1}\phi(w^i)\right)\cdot\left(\sum_{i=1}^n \phi(a_1b_1\ldots a_i) - \phi(a_1b_1\ldots a_ib_i)\right)\right) = \Ann_{RQ}\left(\sum_{i=0}^{k-1}\phi(w^i)\right)
\]
In particular, this holds if $RQ$ is a domain and $\phi(u)\neq 1$ for all proper non-empty subwords $u$ of $w$. Then the following map
    \[
    \left(\bigoplus_{AhK}H_1(K\cap A^h, R)\right)\oplus\left(\bigoplus_{BhK}H_1(K\cap B^h, R)\right)\to H_1(K, R)
    \]
induced by inclusion is injective.
\end{lemma}

\begin{proof}
Let $X_A$ be a $2$-complex with $\pi_1(X_A) = A$ and let $X_B$ be a $2$-complex with $\pi_1(X_B) = B$. Let $X = X_A\sqcup X_B$ and let $Z$ be the $2$-complex obtained by attaching a 1-cell $e$ to $X_A\sqcup X_B$, with one end in $X_A$ and the other in $X_B$ and a $2$-cell $c$ along a boundary path spelling out the word $w^k$. We have $\pi_1(Z) = G$ by construction. Let $p\colon Y\to Z$ be the cover with $\pi_1(Y) = K$ and hence $\deck(p) = Q$. Denote by $p^{-1}(X) = \overline{Y}$. Let $\tilde{e}\subset Y$ be a lift of $e$ and. Let $\tilde{c}$ be a lift of $c$ and let $\tilde{e}\subset Y$ be the lift of $e$ lying in the boundary of $\tilde{c}$ between the subpath corresponding to $b_n$ and the subpath corresponding to $a_1$. Then we have $C_1(Y, R) \isom C_1(\overline{Y}, R)\oplus RQ\cdot [\tilde{e}]$ and $C_2(Y, R) \isom C_2(\overline{Y}, R)\oplus RQ\cdot [\tilde{c}]$. 

A simple computation shows that there is some $f\in C_1(\overline{Y}, R)$ such that
\[
\partial_2([\tilde{c}]) = \left(\sum_{i=0}^{k-1}\phi(w^i)\right)\cdot \left(f + \left(\sum_{i=1}^n \phi(a_1b_1\ldots a_i) - \phi(a_1b_1\ldots a_ib_i)\right)\cdot [\tilde{e}]\right)
\]
 By assumption, we have that 
\[
\Ann_{RQ}\left(\left(\sum_{i=0}^{k-1}\phi(w^i)\right)\left(\sum_{i=1}^n \phi(a_1b_1\ldots a_i) - \phi(a_1b_1\ldots a_ib_i)\right)\right) = \Ann_{RQ}(\partial_2([\widetilde{c}]).
\]
If $m\in C_2(Y, R)$, then we may write $m = m' + r\cdot [\widetilde{c}]$ for some $m'\in C_2(\overline{Y}, R)$ and some $r\in RQ$. Then $\partial_2(m) = \partial_2(m') + r\cdot \partial_2([\widetilde{c}])$ and so $\partial_2(m)\in C_1(\overline{Y}, R)$ if and only if $r\in \Ann_{RQ}(\partial_2([\widetilde{c}]))$ by the above. In particular, $\Ima(\partial_2)\cap C_1(\overline{Y}, R) = \partial_2(C_2(\overline{Y}, R))$ and so 
\[
\frac{Z_1(\overline{Y}, R)}{\partial_2(C_2(\overline{Y}, R))} = \frac{\ker(\partial_1)\cap C_1(\overline{Y}, R)}{\Ima(\partial_2)\cap C_1(\overline{Y}, R)}\leqslant H_1(K, R).
\]
Denoting by $\overline{Y} = p^{1}(X_A) \sqcup p^{-1}(X_B) = \overline{Y}_A\sqcup\overline{Y}_B$, we have
\begin{align*}
\frac{Z_1(\overline{Y}_A, R)}{\partial_2(C_2(\overline{Y}_A, R))} &= \bigoplus_{AhK}H_1(K\cap A^h, R)\\
\frac{Z_1(\overline{Y}_B, R)}{\partial_2(C_2(\overline{Y}_B, R))} &= \bigoplus_{BhK}H_1(K\cap B^h, R)
\end{align*}
and so the result holds.
\end{proof}

\begin{lemma}
\label{tree_injection}
Let $\mathbb{T}$ be a graph of groups with underlying graph a tree $T$ and with $G = \pi_1(\mathbb{T})$. Let $R$ be a ring, $M$ a left $RG$-module and suppose that for all edges $e\subset T$, the edge group inclusion $T_e\injects T_v$ induce an injection in homology $H_1(T_e, M)\injects H_1(T_v, M)$. Then the map $H_1(T_v, M)\to H_1(G, M)$ induced by inclusion is also injective for all vertices $v\in T$.
\end{lemma}

\begin{proof}
We assume that each edge in $T$ is oriented. Denote by $E(T)$ the set of oriented edges of $T$ and by $V(T)$ the set of vertices. Denote by $e^{-}$ the origin vertex and by $e^+$ the target vertex of the edge $e$. The edge inclusions are denoted by $\partial_e^{\pm}\colon T_e\to T_{e^{\pm}}$. The maps on homology induced by inclusion are then denoted by $(\partial_e^{\pm})_{\#}\colon H_1(T_e, M)\to H_1(T_{e^{\pm}}, M)$.

Let us inspect the following extract from the Mayer--Vietoris sequence induced by the graph of groups structure (see work of Chiswell \cite{Ch76}):
\[
\bigoplus_{e\in E(T)}H_1(T_e, M)\rightarrow \bigoplus_{v\in V(T)}H_1(T_v, M) \rightarrow H_1(G, M).
\]
The fist map is $\sum_{e\in E(T)}((\partial_e^-)_{\#} + (\partial_e^+)_{\#})$. In order to complete the proof, we need to show that 
\[
\Ima\left(\sum_{e\in E(T)}((\partial_e^-)_{\#} + (\partial_e^+)_{\#})\right)\cap H_1(T_v, R) = 0
\]
for all $v\in V(T)$ as this will imply the second map is injective on each factor. 

Let $v\in V(T)$, let $0\neq m\in \bigoplus_{e\in E(T)}H_1(T_e, R)$ and denote by $m_e$ the projection of $m$ to the $H_1(T_e, R)$ component. Let $S\subset T$ be the union of all edges $e$ such that $m_e\neq 0$. Since $T$ is a tree and $S$ is compact, it follows that there is some edge $e\subset S$ such that $e^-$ (or $e^+$) is distinct from $v$ and distinct from all other endpoints of edges in $S$. Thus, the projection of
\[
\left(\sum_{e\in E(T)}((\partial_e^-)_{\#} + (\partial_e^+)_{\#})\right)(m) = \sum_{e\in E(T)}((\partial_e^-)_{\#}(m_e) + (\partial_e^+)_{\#}(m_e))
\]
to $H_1(T_{e^-}, M)$ (or $H_1(T_{e^+}, M)$) is non-zero. In particular, the above element cannot lie in $H_1(T_v, M)$ and so the proof is complete.
\end{proof}

\begin{proof}[Proof of \cref{perfect_subgroup}]
By \cref{amalgam}, we have
    \[
    G \isom A\underset{A_0}{*}G_0\underset{B_0}{*}B
    \]
where $A_0$ and $B_0$ are the subgroups generated by the elements in $A$ and $B$ appearing in $w$ respectively and where $G_0\isom \frac{A_0*B_0}{\normal{w^k}}$. Now $K$ acts on the Bass--Serre tree for this amalgam. Since $K\cap A = 1 = K\cap B$, the quotient graph of groups for $K$ has trivial edge groups and each non-trivial vertex group is conjugate into $G_0$. Hence, $K$ splits as a free product of subgroups of conjugates of $K\cap G_0$ and a free group. Since $H_1(K, R) = 0$, we actually must have that $K$ splits as a free product of conjugates of $K_0 = K\cap G_0$. This implies that $H_1(K_0, R) = 0$. Hence, if we can show that the theorem holds for $K_0\triangleleft G_0$, we will have shown that it also holds for $K\triangleleft G$. We note that all the same hypotheses still hold for $K_0\triangleleft G_0$.

The proof now proceeds by induction on the word length of $w$. By \cref{HNN}, $G_0$ splits as a HNN-extension $H*_{\psi}$ where $H \isom \frac{J_0*\Z}{\normal{u^k}} \isom \frac{\Z*J_1}{\normal{v^k}}$ with $\psi$ identifying $J_0$ with $J_1$ and such that the lengths of $u$ and $v$ are strictly smaller than the length of $w$. Importantly, since $u$ and $v$ are equal to $w$ as words over $A_0*B_0$, we have that $\phi(x)\neq 1$ for all proper non-empty subwords $x$ of $u$ and $v$. We will now show that $K_0$ splits as a free product of conjugates of $K_0\cap H$.

Now $K_0$ acts on the Bass--Serre tree for the splitting $G_0\isom H*_{\psi}$. The fundamental group of a graph of groups surjects the fundamental group of its underlying graph. Since $H_1(K_0, R) = 0$, this implies that the underlying graph of the quotient graph of groups for $K_0$ has trivial fundamental group and so is a tree. By \cref{relator_form}, the hypotheses of \cref{homology_injection} are satisfied for the one-relator product $H = \frac{J_0*\Z}{\normal{u^k}} \isom \frac{\Z*J_1}{\normal{v^k}}$ with normal subgroup $K_0\cap H$. Hence, applying \cref{homology_injection}, we see that the maps on homology over $R$ from edge groups to adjacent vertex groups in this graph of groups are injective. By \cref{tree_injection}, since $H_1(K_0, R) = 0$, the first homology of the edge groups over $R$ must vanish. Since each edge group is a subgroup of a conjugate of $J_0$ or $J_1$, by \cref{HNN} each edge group splits as a free product of subgroups of conjugates of $A$ and $B$ and a free group. Since $K_0\cap A = 1 = K_0\cap B$, this implies that each edge group is free. Since each edge group is $R$-perfect, it follows that each edge group must actually be trivial. Thus, $K_0$ splits as a free product of conjugates of $K_0\cap H$ as claimed. 

In the base case, $J_0 = J_1 = 1$ and $H = \langle w\rangle\isom \Z/k\Z$. Since $RQ$ is a domain, $Q$ is torsion-free and so $H = K_0\cap H$. For the inductive step we have $K_0\cap H = \normal{w}\cap H$. Since $K$ splits as a free product of conjugates of $K_0$, we then have that $K = \normal{w}$. Finally, since $H_1(\Z/k\Z, R) = 0$ precisely when $k$ is a unit in $R$, the result follows.
\end{proof}

We may slightly improve \cref{perfect_subgroup} and \cref{lift_criterion} for one-relator groups.

\begin{corollary}
\label{perfect_subgroup_1-rel}
Let $R$ be a domain, $F$ a free group, a word $w\in F$ that is not a proper power and consider the one-relator group $G = F/\normal{w^k}$ where $k\in \N$. Let $K\trianglelefteq G$ be a normal subgroup such that, if $\phi\colon G\to G/K = Q$ denotes the quotient map, then $RQ$ is a domain and $\phi(u)\neq 1$ for all proper non-empty subwords $u$ of $w$. If $H_1(K, R) = 0$, then $K= \normal{w}$ and $k$ is a unit in $R$.
\end{corollary}

\begin{proof}
If $w\in F$ is a primitive element, then $F/\normal{w^k}\isom F'*\Z/k\Z$ for some free group $F'$ and where $\Z/k\Z$ is generated by the image of $w$. In this case, the result is clear. If $w\in F$ is not primitive, then we may express $F$ as a free product of two free groups $A*B$ such that $w$ is a cyclically reduced word over $A*B$ of length at least two. By \cref{homology_injection}, we have that $H_1(K\cap A, R)$ and $H_1(K\cap B, R)$ inject into $H_1(K, R) = 0$ which implies that $H_1(K\cap A, R) = 0 = H_1(K\cap B, R)$. Since non-trivial subgroups of free groups have non-trivial first homology, this implies that $K\cap A = 1 = K\cap B$ and we may complete the proof by applying \cref{perfect_subgroup}.
\end{proof}

Using \cref{perfect_subgroup_1-rel} in place of \cref{perfect_subgroup}, we may slightly improve \cref{lift_criterion} for groups with cyclic relation module. The proof is identical to that of \cref{lift_criterion}, although we may use Weinbaum's theorem \cite{We72} in place of Howie's.

\begin{corollary}
\label{lift_criterion_1-rel}
Let $R$ be a ring, let $Q = F/N$ be a group such that $RQ$ is a domain and such that $N_{R} = (r)$ for some $r\in N_{R}$. Let $w\in N$ be a cyclically reduced element such that the image of $w$ in $N_{R}$ is $r$. Then $N = \normal{w}$ if and only if no proper non-empty subword of $w$ lies in $N$.
\end{corollary}

\cref{Weinbaum_converse} from the introduction is \cref{lift_criterion_1-rel} in the special case $R = \Z$.

\section{Admissibility results}
\label{sec:admissibility}

In this section, we shall use \cref{perfect_subgroup} to prove all our admissibility results. Since our main results will account for rings other than $\Z$ and will actually be relativisations of the results stated in the introduction, we will need to set up some further notation and definitions first.

Let $X$ be a $G$-$2$-complex and let $A_*\oplus C_*(X, R)$ be a complex of free $RG$-modules and let $B_n\subset A_n$ be free bases for each $n\geqslant 0$. We say that $A_*\oplus C_*(X, R)$ is \emph{$R$-admissible relative to $X$} if there is a $G$-$2$-complex $Z$, containing $X$, and an $RG$-isomorphism $C_*(Z, R)\to A_* \oplus C_*(X, R)$ which is the identity on $C_*(X, R)\subset C_*(Z, R)$ and which sends natural free basis elements for $C_n(Z, R)$ to the chosen free basis elements in $B_n$. When $X = \emptyset$, we say that $A_*$ is \emph{$R$-admissible}. When $R = \Z$, we say $A_*\oplus C_*(X)$ is \emph{admissible relative to $X$}. When $R = \Z$ and $X = \emptyset$, this coincides with the definition of admissibility given in the introduction. We say $A_*\oplus C_*(X, R)$ is \emph{weakly} $R$-admissible relative to $X$ if it is $R$-admissible for some choice of free basis $B_2\subset A_2$. Other variations of weak admissibility are defined in the same way as the variations of $R$-admissibility.

We say that the complex $A_*\oplus C_*(X, R)$ with differentials $d_n$ is \emph{algebraic relative to $X$} if the following holds:
\begin{enumerate}
\item $d_n\mid C_n(X, R) = 1\otimes\partial_n\colon C_n(X, R) \to C_{n-1}(X, R)$ for each $n$,
\item $H_1(A_*\oplus C_*(X, R)) = 0$ and $H_0(A_*\oplus C_*(X, R)) = R$,
\item for each $E\in B_1$, we have $d_1(E) = g_1\cdot V_1 - g_0\cdot V_0$ where $g_i\in G$ and $V_i$ either lies in $B_0$ or is a natural basis element of $C_0(X, R)$ for $i = 0, 1$.
\end{enumerate}
If $X = \emptyset$ and $A_*$ satisfies the last two conditions above, we simply say $A_*$ is \emph{algebraic}.

\subsection{Weak admissibility over a domain}

The main technical result from which all our main results will be derived is the following.

\begin{theorem}
\label{R-relative_main}
Let $R$ be a ring, let $G$ be a group and let $X$ be a $G$-$2$-complex with $H_1(X, R) = 0$. Let
\[
\begin{tikzcd}
A_2\oplus C_2(X, R) \arrow[r, "d_2"] & A_1\oplus C_1(X, R) \arrow[r, "d_1"] & A_0\oplus C_0(X, R)
\end{tikzcd}
\]
be an algebraic $RG$-$2$-complex relative to $X$. Let
\begin{align*}
B_0 &= \{V_1, V_2, \ldots\}\subset A_0,\\
B_1 &= \{E_1, E_2, \ldots\}\subset A_1,\\
B_2 &= \{C_1, C_2, \ldots\}\subset A_2,
\end{align*}
be free $RG$-bases and suppose that the induced map $A_2\to A_1$ is given by a lower trapezoidal matrix $M_2$ over the given bases. 

If $RG$ is a domain and the $ij_i$-entry of $M_2$ is not left engulfing for each $i$, then there is a $G$-$2$-complex $Z$ and an $RG$-isomorphism $f$ of complexes
\[
\begin{tikzcd}
{C_2(Z, R)} \arrow[r, "1\otimes\partial_2"] \arrow[d, "f_2"] & {C_1(Z, R)} \arrow[r, "1\otimes\partial_1"] \arrow[d, "f_1"] & {C_0(Z, R)}  \arrow[d, "f_0"]\\
A_2\oplus C_2(X, R) \arrow[r, "d_2"] & A_1\oplus C_1(X, R) \arrow[r, "d_1"] & A_0\oplus C_0(X, R) 
\end{tikzcd}
\]
with the following properties:
\begin{enumerate}
\item $X\subset Z$ and $f\mid C_*(X, R)$ is the identity.
\item $f_0$ and $f_1$ send $G$-orbit representatives of $0$-cells and $1$-cells in $Z-X$ to basis elements in $B_0$ and $B_1$. Moreover, $Z-X$ contains a $G$-orbit of $2$-cells for each $C_l$ and $f_2$ sends an orbit representative of the $l^{\text{th}}$ $2$-cell in $Z-X$ to an element of the form 
\[
r_l\cdot C_l + \sum_{i = 1}^{l-1}r_i\cdot C_i + x
\]
where $x\in C_2(X, R)$, each $r_i\in RG$ and where $r_l\in R$ is a unit.
\item The attaching maps of $2$-cells in $Z-X$ are all embeddings.
\item $G\backslash Z$ is reducible to $G\backslash X$.
\item If each component of $X$ is simply connected and each component of $G\backslash X$ has locally indicable fundamental group, then $Z$ is simply connected and $G\isom \pi_1(G\backslash Z)$ is locally indicable.
\end{enumerate}
\end{theorem}

We have the following immediate corollary.

\begin{corollary}[Weak relative $R$-Admissibility]
\label{weak_relative_R-admissibility}
Let $R$ be a ring, let $G$ be a group with $RG$ a domain and let $X$ be a $G$-CW-complex such that each component of $X$ is simply connected and each component of $G\backslash X$ has locally indicable fundamental group. Let
\[
\begin{tikzcd}
A_2\oplus C_2(X, R) \arrow[r, "d_2"] & A_1 \oplus C_1(X, R) \arrow[r, "d_1"] & A_0\oplus C_0(X, R)
\end{tikzcd}
\]
be an algebraic $RG$-complex relative to $X$ with chosen free bases for each $A_i$. If the map $A_2\to A_1$ induced by $d_2$ can be given by a lower trapezoidal matrix over the given bases, such that the $ij_i$ entries are all non left engulfing, then $A_*\oplus C_*(X, R)$ is weakly $R$-admissible relative to $X$.
\end{corollary}

Using the fact that $RG$ contains no engulfing elements if $R$ is a domain and $G$ is right orderable by \cref{ro}, \cref{weak_relative_R-admissibility} implies the following.

\begin{corollary}[Weak $R$-Admissibility]
\label{weak_R-admissibility}
Let $R$ be a domain, let $G$ be a right orderable group and let
\[
\begin{tikzcd}
A_2 \arrow[r, "d_2"] & A_1 \arrow[r, "d_1"] & A_0
\end{tikzcd}
\]
be an algebraic $RG$-complex with chosen free bases for each $A_i$. If the map $d_2\colon A_2\to A_1$ can be given by a lower trapezoidal matrix over the given bases, then $A_*$ is weakly $R$-admissible.
\end{corollary}

Before beginning the proof of \cref{R-relative_main}, we shall need a technical lemma.

\begin{lemma}
\label{inductive_lemma}
Let $G$ be a group, let $X$ be a $G$-$2$-complex and consider the $RG$ map $1\otimes \partial_1\colon C_1(X, R)\to C_0(X, R)$. Let $h\subset X$ be a 1-cell and denote by $K = \ker(1\otimes \partial_1)\cap p_h^{-1}(0)$ where $p_h\colon C_1(X, R)\to RG\cdot [h]$ denotes the projection map. If $r\in \ker(1\otimes \partial_1)$ is such that $0\neq p_h(r)$ is left non-engulfing and such that $\ker(1\otimes \partial_1) = RG\cdot r + K$, then there is an embedded cycle $\lambda\colon S^1\injects X^{(1)}$ and a unit $u\in R$ such that $r\in (u\otimes[\lambda]) + K$.
\end{lemma}

\begin{proof}
Let $E\cup\{h\}$ be a collection of $G$-orbit representatives of 1-cells in $X$ and let $V$ be a collection of $G$-orbit representatives of 0-cells so that
\begin{align*}
C_0(X, R) &\isom \bigoplus_{v\in V}RG\cdot [v]\\
C_1(X, R) &\isom RG\cdot [h]\oplus\bigoplus_{e\in E}RG\cdot [e]
\end{align*}
Denote by $\Gamma\subset X^{(1)}$ the minimal subcomplex such that $r\in Z_1(\Gamma, R)$. Note that $\Gamma$ does not contain any vertices of degree one or less and is not necessarily connected. We claim that:
\begin{equation}
\label{key_equality}
Z_1(\Gamma, R) = R\cdot r + K\cap Z_1(\Gamma, R).
\end{equation}
Let $\delta\in Z_1(\Gamma, R)$. Since $Z_1(\Gamma, R)\leqslant Z_1(X, R) = RG\cdot r + K$, there is some $t\in RG$ and some $k\in K$ such that $t\cdot r + k = \delta$. By definition of $\delta$, we then must have that $\Gamma$ contains the 1-cells in 
\[
\supp(t\cdot p_h(r))\cdot [h] = \{ g\cdot h \mid g\in \supp(t\cdot p_h(r))\}. 
\]
By definition of $\Gamma$, the only 1-cells in the $G$-orbit of $h$ that lie in $\Gamma$ are those in $\supp(p_h(r))\cdot h$. Thus, we have $\supp(t\cdot p_h(r)) \subset \supp(p_h(r))$ and so, since $p_h(r)$ was assumed to be left non-engulfing, we have $\supp(t) = 1$. In other words, $t\in R$ and so $\delta\in R\cdot r + K\cap Z_1(\Gamma, R)$ as claimed.

If $\Gamma_h\subset \Gamma$ is the subcomplex obtained by removing all $G$-orbits of the 1-cell $h$, then $K\cap Z_1(\Gamma, R) = Z_1(\Gamma_h, R)$ and $Z_1(\Gamma, R) \isom R\oplus Z_1(\Gamma_h, R)$ by \cref{key_equality} and the fact that $Z_1(\Gamma, R)\neq Z_1(\Gamma_h, R)$. Hence, there is some embedded cycle $\lambda\colon S^1\injects \Gamma$ traversing some $g\cdot h$ for some $g\in G$, such that $Z_1(\Gamma, R) \isom R\cdot(1\otimes[\lambda])\oplus Z_1(\Gamma_h, R)$. Since the projection of $R\cdot r + Z_1(\Gamma_h, R) = Z_1(\Gamma, R) = R\cdot (1\otimes[\lambda]) + Z_1(\Gamma_h, R)$ to the $RG\cdot [h]$ factor is generated by $R\cdot p_h(r)$ and also by the image of $R\cdot(1\otimes[\lambda])$, we see that there is a unit $u\in R$ such that $r\in (u\otimes [\lambda]) +Z_1(\Gamma_h, R) \subset (u\otimes [\lambda]) + K$ as claimed.
\end{proof}

\begin{proof}[Proof of \cref{relative_main}]
Denote by $n_i = |B_i|\in \N\cup\{\infty\}$. We have $A_0 = \bigoplus_{j=1}^{n_0}\Z G\cdot V_j$, $A_1 = \bigoplus_{j=1}^{n_1}\Z G\cdot E_j$ and $A_2 = \bigoplus_{j=1}^{n_2}\Z G\cdot C_j$. By assumption, for each $E_j$, there are elements $g_{j, 0}, g_{j, 1}\in G$ and elements $v_0, v_1$ which are either 0-cells in $Y$ or basis elements in $B_0$, such that $d_1(E_j) = g_{j, 1}\cdot v_1 - g_{j, 0}\cdot v_0$. Denote by $M_2$ the $RG$-matrix such that the map $A_2\to A_1$ is given by right multiplication by $M_2$. By assumption, for each $i\in [n_2]$, there is an integer $j_i$ such that $C_i\cdot M_2 = (m_{i0}, m_{i1}, \ldots, m_{ij_i}, 0, \ldots)$ with $m_{ij_i}$ not left engulfing and $m_{ij} = 0$ for all $j>j_i$. 

We now construct an oriented $G$-1-complex $Y$. The 0-cells of $Y$ are in bijection with the 0-cells in $X$ and the $G$-orbits of each $V_j$ and the 1-cells are in bijection with the 1-cells in $X$ and the $G$-orbits of the $E_j$. The oriented 1-cell corresponding to $g\cdot E_j$ has origin 0-cell $g_{j, 0}\cdot v_0$ and target 0-cell $g_{j, 1}\cdot v_1$. For each $i\in [n_1]$, denote by $Y_i\subset Y$ the subcomplex consisting of $X^{(1)}$ together with the $G$-orbits of the basis 1-cells from $\{E_1, \ldots, E_i\}$ and the $G$-orbits of the adjacent 0-cells. Since $H_0(A_*\oplus C_*(X, R)) \isom R$, the CW-complex $Y$ is connected and so
\[
Y = \bigcup_{i = 1}^nY_i.
\]
Denote by $d_1^i\colon C_1(Y_i, R) \to C_0(Y_i, R)$ the boundary map restricted to the subcomplex $Y_i$. Note that the following diagram commutes and has exact rows
\[
\begin{tikzcd}
{C_1(Y_i, R)} \arrow[r, "d_1^i"] \arrow[d, hook]            & {C_0(Y_i, R)} \arrow[d, hook]                                    &                                            &   \\
{C_1(Y, R)} \arrow[r, "1\otimes\partial_1"] \arrow[equal]{d} & {C_0(Y, R)} \arrow[r, "\epsilon"] \arrow[equal]{d} & R \arrow[r] \arrow[equal]{d}& 0 \\
A_1 \oplus C_1(X, R)\arrow[r, "d_1"]                                       & A_0\oplus C_0(X, R) \arrow[r, "\epsilon"]                                        & R \arrow[r]                                & 0
\end{tikzcd}
\]
where $\epsilon$ is the augmentation map defined by sending each basis element to $1\in R$. For each $l\in [n_2]$, denote by $K_l = \ker(d_1^l)$, setting $K_0 = \ker(C_1(X, R)\to C_0(X, R))$. For each $l\in [n_1]$, denote by 
\[
I_l = C_2(X, R)\oplus \bigoplus_{i = 1}^lRG\cdot C_i,
\]
setting $I_0 = C_2(X, R)$. We claim that for each $i\in [n_2]$, we have 
\begin{align}
\label{claim}
K_{j_i} &= d_2(I_i)\\
	    &= d_2(RG\cdot C_i) + K_{j_i-1}\\
	     &\isom RG\oplus K_{j_i-1}
\end{align}
The proof proceeds by induction on $i$. Suppose the claim is true for all integers less than $i$ and let $k\in K_{j_i} - K_{j_i - 1} = K_{j_i} - d_2(I_{i-1})$. By exactness, there is some $(r_1, r_2, \ldots, r_l, 0, \ldots)\in A_2$, where $r_l\neq 0$ and $r_p = 0$ for all $p>l$, such that
\[
p_{A_1}(k) = (r_1, \ldots, r_l, 0, \ldots)\cdot M_2 = (r_1, \ldots, r_{l-1}, 0, \ldots)\cdot M_2 + r_{l}\cdot(m_{l1}, \ldots, m_{lj_l}, 0, \ldots)
\]
where $p_{A_1}\colon C_1(X, R)\oplus A_1\to A_1$ denotes the projection to $A_1$. Note that since $m_{lj_l}$ is left non-engulfing, we have that $\Ann_{RG}(m_{lj_l}) = 0$ and so $r_l\cdot m_{lj_l}\neq 0$. Since $m_{pj_l} = 0$ for all $p<l$ and $k\in K_{j_i}$ we must have that $j_l\leqslant j_i$. This implies that $l\leqslant i$ and so that $k\in d_2(I_i)$ as claimed. The fact that $d_2(I_i)= d_2(RG\cdot C_i) + K_{j_i - 1} \isom RG\oplus K_{j_i-1}$ now also readily follows from the fact that $RG$ is a domain and the fact that the projection of $K_{j_i-1}$ to the $j_i^{\text{th}}$ factor $RG\cdot E_{j_i}$ of $A_1$ is $0$.

Using \cref{claim}, for each $i\in [n_2]$, we may now apply \cref{inductive_lemma} to $d_1^{j_i}\colon C_1(Y_{j_i}, R)\to C_0(Y_{j_i}, R)$ with $E_{j_i}$ playing the role of the 1-cell $h$, $K_{j_i - 1}$ playing the role of $K$ and with $d_2(C_i)$ playing the role of $r$. We thus obtain embedded cycles $\lambda_i\colon S^1\injects Y_{j_i}$ and units $u_i\in R$ such that
\[
d_2(C_i) \in (u_i\otimes[\lambda_i]) + d_2(I_{i-1}).
\]

Let $Z$ be the $G$-CW-complex obtained from $Y$ by attaching all cells from $X$ and by attaching $2$-cells along each $\lambda_i$ and their $G$-orbits. The attaching maps of $2$-cells in $Z-X$ are embeddings by definition of the $\lambda_i$. The maps $f_0$ and $f_1$ have already been defined. Define the map $f_2$ as the identity on $2$-cells in $X$ and by sending the $2$-cell with attaching map $\lambda_i$ to $u_i^{-1}\cdot (C_i - \beta_i)$. This fulfills the required properties of the map $f\colon C_*(Z, R) \to C_*(X, R)\oplus A_*$ and establishes the first three properties. 

For each $i\in [n_1]$, denoting by $X_i\subset Z$ the maximal $G$-$2$-subcomplex with $Y_i = X_i^{(1)}$, we claim that $G\backslash X_i\subset G\backslash X_{i+1}$ is an elementary reduction so that $G\backslash Z$ reduces to $G\backslash X$. If $i+1\neq j_{l}$ for any $l\in [n_2]$, then $G\backslash X_{i+1} - G\backslash X_i$ contains no $2$-cells and contains a single 1-cell. Thus, it is an elementary reduction. On the other hand, if $i+1 = j_l$ for some $l\in [n_2]$, then $G\backslash X_{i+1} - G\backslash X_i$ contains precisely one $2$-cell (attached along the projection $\overline{\lambda}_{l}\colon S^1\to G\backslash X_{i+1}$ of $\lambda_{l}$) and precisely one 1-cell (the projection $\overline{E}_{i+1}$ of the 1-cell $E_{i+1}$). Moreover, the attaching map $\overline{\lambda}_{l}$ of the $2$-cell in $G\backslash X_{i+1} - G\backslash X_i$ traverses the 1-cell $\overline{E}_{i+1}\subset G\backslash X_{i+1} - G\backslash X_i$. Since any free homotopy in $G\backslash X_{i+1}$ of $\overline{\lambda}_{l}$ can be lifted to a free homotopy of $\lambda_{l}$ in $X_{i+1}$, in order to complete the proof of the claim, we need to show that $\lambda_{l}$ is not freely homotopic into $X_i$ within $X_i \cup X_{i+1}^{(1)}$. But since $\lambda_{l}$ was an embedding and there are no $2$-cells attached along the $G$-orbit of $E_{i+1}$, any cycle that is freely homotopic to $\lambda_{l}$ in $X_i\cup X_{i+1}^{(1)}$ must still traverse some $g\cdot E_{i+1}$. Hence, $G\backslash X_i\subset G\backslash X_{i+1}$ is an elementary reduction and so $G\backslash Z$ is reducible to $G\backslash X$. This establishes the fourth property.

All that remains to show is that if each component of $X$ is simply connected and each component of $G\backslash X$ has locally indicable fundamental group, then $Z$ is simply connected and $G$ is locally indicable. We prove by induction that each component of $X_i$ is simply connected and that each component of $G\backslash X_i$ has locally indicable fundamental group for each $i\in [n_1]$. From here it follows that $X$ will be simply connected and that $G$ will be locally indicable. The base case of the induction $X_0 = X$ holds by assumption. Now suppose the result is true for all $i<k$ and we wish to prove the claim for $X_k$. Let $\overline{X}_k\subset G\backslash X_k$ be any connected component. There is some subcomplex $\overline{X}_{k-1}\subset G\backslash X_{k-1}$ such that $\overline{X}_{k-1}\subset \overline{X}_k$ is an elementary reduction. There are now two cases to consider, depending on whether $\overline{X}_{k-1}$ is connected or not. Let us suppose that $\overline{X}_{k-1}$ is not connected so that $\overline{X}_{k-1} = \overline{X}_{k-1}'\sqcup \overline{X}_{k-1}''$ with $\overline{X}_{k-1}', \overline{X}_{k-1}''$ both connected. The other case is handled in the same way so we omit it. By induction, since each component of $X_{k-1}$ is simply connected, we have that $A = \pi_1(\overline{X}_{k-1}'), B = \pi_1(\overline{X}_{k-1}'')$ are both locally indicable subgroups of $G$. We have $\pi_1(\overline{X}_k) \isom \frac{A*B}{\normal{w}}$ where $w\in A*B$ is either $1$ or is given by the attaching map of $\lambda_i$ where $j_i= k$. Now consider the homomorphism
\[
\phi\colon\pi_1(\overline{X}_k) \to G
\]
The kernel $\ker(\phi)$ is precisely the fundamental group of a (any) component of $X_k$ mapping to $\overline{X}_k$. Since each $\lambda_i$ is an embedding, no proper non-empty subword $u$ of $w$ satisfies $\phi(u) = 1$. Since $A, B$ are subgroups of $G$, we have $\ker(\phi)\cap A = \ker(\phi)\cap B = 1$. By \cref{claim}, we have that
\[
\begin{tikzcd}
C_2(X_k, R) \arrow[r] & C_1(X_k, R) \arrow[r] & C_0(X_k, R)
\end{tikzcd}
\]
is exact, and hence $H_1(X_k, R) = 0$ and so $H_1(\ker(\phi), R) = 0$. We may now apply \cref{perfect_subgroup} to obtain that $\ker(\phi) = 1$. Hence, each component of $X_k$ is simply connected. Since $G$ is torsion-free and $\pi_1(\overline{X}_k)\leqslant G$, it follows from \cref{inclusion} that $\pi_1(\overline{X}_k)$ is locally indicable. This concludes the proof.
\end{proof}

\subsection{Admissibility over the integers}

When $R = \Z$, we may improve \cref{relative_main} by dropping the weakness. Combined with \cref{wall_theorem}, this establishes the remaining direction of \cref{reducible_intro}, the other direction being \cref{lower_trapezoidal}.

\begin{theorem}
\label{relative_main}
Let $G$ be a group and let $X$ be a $G$-$2$-complex with $H_1(X) = 0$. Let
\[
\begin{tikzcd}
A_2\oplus C_2(X) \arrow[r, "d_2"] & A_1\oplus C_1(X) \arrow[r, "d_1"] & A_0\oplus C_0(X)
\end{tikzcd}
\]
be an algebraic $\Z G$-complex relative to $X$. Let 
\begin{align*}
B_0 &= \{V_1, V_2, \ldots\}\subset A_0,\\
B_1 &= \{E_1, E_2, \ldots\}\subset A_1,\\
B_2 &= \{C_1, C_2, \ldots\}\subset A_2,
\end{align*}
be free $\Z G$-bases and suppose the induced map $A_2\to A_1$ is given by a lower trapezoidal matrix $M_2$ over the given bases.

If each component of $X$ is simply connected, $\Z G$ is a domain and the $ij_i$-entry of $M_2$ is not left engulfing for each $i$, then there is a $G$-$2$-complex $Z$ and a $\Z G$-isomorphism
\[
\begin{tikzcd}
{C_2(Z)} \arrow[r, "1\otimes\partial_2"] \arrow[d, "f_2"] & {C_1(Z)} \arrow[r, "1\otimes\partial_1"] \arrow[d, "f_1"] & {C_0(Z)} \arrow[d, "f_0"]\\
A_2\oplus C_2(X) \arrow[r, "d_2"] & A_1\oplus C_1(X) \arrow[r, "d_1"] & A_0\oplus C_0(X)
\end{tikzcd}
\]
with the following properties:
\begin{enumerate}
\item $X\subset Z$ and $f\mid C_*(X)$ is the identity.
\item $f_i$ sends $G$-orbit representatives of $i$-cells in $Z-X$ to basis elements in $B_i$.
\item $G\backslash Z$ is homotopy equivalent, keeping $G\backslash X$ fixed, to a CW-complex $Z'$  such that $Z'$ is reducible to $G\backslash X$.
\item If each component of $G\backslash X$ has locally indicable fundamental group, then $Z$ is simply connected and $G\isom \pi_1(G\backslash Z)$ is locally indicable. 
\end{enumerate}
\end{theorem}

\begin{proof}
Let $n_i = |B_i| \in \N\cup\{\infty\}$, let $Z'$ be the $G$-$2$-complex and let $f'\colon C_*(Z')\to A_*\oplus C_*(X)$ be the map furnished by \cref{R-relative_main}. For each $i\in [n_1]$, let $X'_{i}\subset X_{i+1}'$ be the inclusion of $G$-$2$-complexes such that $G\backslash X_i'\subset G\backslash X_{i+1}'$ is an elementary reduction for each $i$. In particular, $X'_0 = X$, $X'_{i+1} - X_{i}'$ contains a single $G$-orbit of $2$-cell $G\cdot c'_{l}$ attached along the attaching maps $G\cdot \lambda_{l}'$ if $i+1 = j_{l}$ or a single $G$-orbit of 1-cell otherwise and such that $Z' = \bigcup_{i=0}^{n_1} X_i'$. By \cref{R-relative_main}, we have that 
\[
f_2'([c_l']) = r_l\cdot C_l + \sum_{i=1}^{l-1}r_i\cdot C_i + x_l.
\]
Since here $R = \Z$ and $r_l$ is a unit, we have that $r_l = \pm1$. After possibly changing orientations, we assume that $r_l = 1$ for all $l$.

We now claim that for all $l\in [n]$ there is a cycle $\lambda_l\colon S^1\to Z'^{(1)}$ such that
\[
[\lambda_l] = [\lambda_l'] - \sum_{i=1}^{l-1}r_i\cdot [\lambda_{i}] - \partial_2(x_l)\in C_1(Z').
\]
and such that $\lambda_l$ is freely homotopic to $\lambda_l'$ within $Z'^{(1)}\cup X'_{j_l-1}$. The proof of the claim proceeds by induction on $l$. In the base case $l = 1$, since each component of $X$ is simply connected, there is some $\lambda_1\colon S^1\to Z'^{(1)}$ which is freely homotopic to $\lambda_1'$ within $X\cup Z'^{(1)}$ and such that $[\lambda_1] = [\lambda_1'] - \partial_2(x_1)$. Now suppose the claim is true for all $\lambda_i$ with $i<l$. It follows by induction that if $i<l$, then $\lambda_i$ is nullhomotopic in $X'_{j_l-1}\cup Z'^{(1)}$ since it is freely homotopic to the boundary of a $2$-cell, $\lambda_i'$. Hence, there is some $\lambda_l\colon S^1\to Z'^{(1)}$ which is freely homotopic to $\lambda_l'$ within $X'_{j_l-1}\cup Z'^{(1)}$ and such that $[\lambda_l] = [\lambda_l'] - \sum_{i=1}^{l-1}r_i\cdot [\lambda_{i}] - \partial_2(x_l)$ as required.

Now let $Z$ be the $G$-CW-complex obtained from $Z'$ by removing all $2$-cells in $Z'-X$ and attaching $2$-cells along $G$-orbits of each $\lambda_l$ instead. By construction, the map $f\colon C_*(Z)\to A_*\oplus C_*(X)$ given by $f'$ on $C_*(X\cup Z^{(1)})\leqslant C_*(Z)$ and by sending the $2$-cell attached along $\lambda_l$ to $C_l$ is a $\Z G$-isomorphism.

Since $\lambda_l$ is freely homotopic to $\lambda_l'$ within $Z^{(1)}\cup X'_{j_l-1}$ by construction, we have that $G\backslash Z$ is homotopy equivalent to $G\backslash Z'$, keeping $G\backslash X$ fixed. Moreover, $G\backslash Z'$ is reducible to $G\backslash X$ and so we have proven the third statement.

We finally prove the fourth statement. By \cref{R-relative_main}, if $X$ is simply connected and each component of $G\backslash X$ has locally indicable fundamental group, then $G$ is locally indicable and $Z'$ is simply connected. Since $Z'$ is homotopy equivalent to $Z$, also $Z$ is simply connected as required.
\end{proof}

\begin{corollary}[Relative Admissibility]
Let $G$ be a group and let $X$ be a $G$-CW-complex such that each component of $X$ is simply connected and each component of $G\backslash X$ has locally indicable fundamental group. Let
\[
\begin{tikzcd}
A_2\oplus C_2(X) \arrow[r, "d_2"] & A_1 \oplus C_1(X) \arrow[r, "d_1"] & A_0\oplus C_0(X)
\end{tikzcd}
\]
be an algebraic $\Z G$-complex relative to $X$ with chosen free bases for each $A_i$. If $\Z G$ is a domain and the map $A_2\to A_1$ induced by $d_2$ can be given by a lower trapezoidal matrix over the given bases, such that the $ij_i$ entries are all non left engulfing then $A_*\oplus C_*(X)$ is admissible relative to $X$.
\end{corollary}

The following is an immediate corollary of \cref{relative_main} and is \cref{main_intro} from the introduction.

\begin{corollary}[Absolute Admissibility]
\label{absolute_admissibility}
Let $G$ be a right orderable group and let 
\begin{equation*}
\begin{tikzcd}
A_2 \arrow[r, "d_2"] & A_1 \arrow[r, "d_1"] & A_0
\end{tikzcd}
\end{equation*}
be an algebraic $\Z G$-complex with chosen free bases for each $A_i$. If the map $d_2$ can be given by a lower trapezoidal matrix over the given bases, then $A_*$ is admissible.
\end{corollary}

\cref{reducible_intro} now follows by combining \cref{wall_theorem} with \cref{lower_trapezoidal} and \cref{absolute_admissibility}.

\section{Applications to relation modules}

In this section we derive consequences for relation modules from the results in \cref{sec:admissibility}. Recall that if $F$ is a free group and $N\triangleleft F$ is a normal subgroup, we say $F/N$ is a presentation for $G = F/N$ and its relation module is the left $\Z G$-module $N_{\ab} = N/[N, N]$, where the $G$ action is given by conjugation. Recall also that if $R$ is a ring, the $R$-relation module for the presentation is the left $RG$-module 
\[
N_R = R\otimes_{\Z}N_{\ab}.
\]

\begin{theorem}
\label{lifting_relations}
Let $R$ be a domain, let $F = F_A*F_B$ be a free group, $N\triangleleft F$ a normal subgroup and let $W_A\subset F_A$, $W_B\subset F_B$ be elements such that $N_R$ is generated as an $RG$-module by the images of $W_A$, $W_B$ and by one other element, where $G = F/N$.

If $G$ is right orderable and $F_A/\normal{W_A}$, $F_B/\normal{W_B}$ are locally indicable and embed in $G$, then there is some $w\in N$ such that $N = \normal{w, W_A, W_B}$ and the following hold:
\begin{enumerate}
\item If $q\in N_R$ is an element that generates $N_R$ together with the images of $W_A, W_B$, then there is some $x\in N_R$ contained in the $RG$-submodule generated by the images of $W_A, W_B$ and, a unit $u\in R$ and some $f\in F$ such that
\[
q = u\otimes w^f[N, N] + x.
\]
\item If $v\in N$ is an element whose image in $N_R$ generates $N_R$ together with $W_A, W_B$, then there is some $y\in N_{\ab}$ contained in the $\Z G$-submodule generated by the images of $W_A, W_B$, an integer $k$ such that $1\otimes k\in R\otimes_{\Z}\Z$ is a unit and some $f\in F$ such that
\[
1\otimes v[N, N] = 1\otimes (w^f)^k[N, N] + 1\otimes y.
\]
\end{enumerate}
In particular, $G =\frac{A*B}{\normal{w}}$ is a one-relator product of $A = F_A/\normal{W_A}$ and $B = F_B/\normal{W_B}$.
\end{theorem}

\begin{proof}
Let $S_A\subset F_A$ and $S_B\subset F_B$ be free generating sets. Let $X_A$ be a presentation complex for $A = F_A/N_A$ with 1-cells in correspondence with elements of $S_A$ and $2$-cells in correspondence with elements of $W_A$. Define $X_B$ similarly. Let $Z$ be the $2$-complex obtained by attaching a 1-cell to $X = X_A\sqcup X_B$ connecting the two components and let $\widetilde{Z} \to Z$ be the regular $G$-cover with $\pi_1(\widetilde{Z}) = \ker(A*B\to G)$. Let $\overline{X}\subset \widetilde{Z}$ be the preimage of $X$ under the cover. Since $C_i(\widetilde{Z}, R) = C_i(\overline{X}, R)$ for $i = 0, 2$ and $C_1(\widetilde{Z}, R) = C_1(\overline{X}, R)\oplus RG$, we obtain the chain complex
\[
C_2(\overline{X}, R) \rightarrow C_1(\overline{X}, R)\oplus RG \xrightarrow{d_1} C_0(\overline{X}, R).
\]
Now $\ker(d_1) \isom N_R$ by \cref{submodule}. By assumption, there is some $1\neq q\in N_R$ such that the image of $C_2(\overline{X}, R)$ together with $q$ generate $\ker(d_1)$. In particular, the following complex
\[
C_2(\overline{X}, R)\oplus RG \xrightarrow{d_2} C_1(\overline{X}, R)\oplus RG \xrightarrow{d_1} C_0(\overline{X}, R).
\]
given by sending $1\in RG$ to $q$ under $d_2$ is exact. By construction, this complex is an algebraic $RG$-complex relative to $\overline{X}$. Since $A$ and $B$ embed in $G$, each component of $\overline{X}$ is simply connected and each component of $G\backslash \overline{X} = X_A\sqcup X_B$ has locally indicable fundamental group. Since group rings of right orderable groups over domains are domains and do not contain engulfing elements by \cref{ro}, we may now apply \cref{weak_relative_R-admissibility} to conclude that there is some $w\in N$ such that $N = \normal{w, W_A, W_B}$.

Now suppose instead that $d_2$ send $1\in RG$ to $1\otimes w[N, N]$. Denoting by $p\colon C_1(\overline{X}, R)\oplus RG \to RG$ the projection map, if $(x, \alpha)\in C_2(\overline{X}, R)\oplus RG$ is such that $d_2(x, \alpha) = 0$, then $p\circ d_2(x, \alpha) = d_2(x, 0) + \alpha\cdot(p\circ d_2(0, 1)) = 0$. Since $p\circ d_2(x, 0) = 0$ and $p\circ d_2(0, 1)\neq 0$, it follows that $\alpha = 0$ as $RG$ is a domain. Thus, we have 
\[
N_R\isom d_2(C_2(\overline{X}, R))\oplus RG.
\]
Let $\overline{W}\subset N_R$ denote the images of $W_A\sqcup W_B$ in $N_R$. Note that $d_2(C_2(\overline{X}, R))$ is precisely the submodule of $N_R$ generated by $\overline{W}$. By \cref{domain_remark}, every unit in $RG$ is of the form $u\cdot g$ where $u\in R$ is a unit and where $g\in G$. Since $1\otimes w[N, N], \overline{W}$ generate $N_R$, it follows that every other element $q\in N_R$ such that $q, \overline{W}$ generate $N_R$ must be of the form $y + u\cdot g\cdot (1\otimes w[N, N]) = y + u\otimes fwf^{-1}[N, N]$ for some $x\in d_2(C_2(\overline{X}, R))$ and some $f\in F$ such that $fN = g$. This yields the first statement.

Now let $v\in N$ be an element such that $1\otimes v[N, N], \overline{W}$ generate $N_R$. By the above, we have $1\otimes v[N, N] = u\otimes fwf^{-1}[N, N] + x$. Since $p(1\otimes v[N, N])=p(u\otimes fwf^{-1}[N, N])$ lies in the image of $\Z G$ under the map $\Z G\to R\otimes_{\Z}\Z G = RG$, we have that $u$ must lie in the image of $\Z$ under the map $\Z \to R\otimes_{\Z}\Z$. Thus, there is some integer $k\neq 0$ such that $u\otimes fwf^{-1}[N, N] = 1\otimes fw^kf^{-1}[N, N]$. Finally, since $x = 1\otimes v[N, N] - 1\otimes fw^kf^{-1}[N, N]$ lies in the image of $N_{\ab}$ under the map $N_{\ab}\to N_R$, we have that $x$ lies in the image of the submodule of $N_{\ab}$ generated by the images of $W_A\cup W_B$.
\end{proof}

The following corollary is \cref{lifting_relations} reformulated in the case $R = \Z$.

\begin{corollary}
\label{Zlifting}
Let $F_A$ and $F_B$ be free groups, let $N\triangleleft F_A*F_B$ be a normal subgroup and suppose that $G = (F_A*F_B)/N$ is right orderable. Let $W_A\subset F_A\cap N$, $W_B\subset F_B\cap N$ be sets such that $A = F_A/\normal{W_A}$ and $B = F_B/\normal{W_B}$ are locally indicable and embed in $G$. If $r\in N_{\Z}$ is such that the images of $W_A, W_B$ in $N_{\Z} = N/[N, N]$ together with $r$ generate $N_{\Z}$ as a left $\Z G$-module, then there is an element $w\in N$ such that:
\[
w[N, N] = r, \quad N = \normal{w, W_A, W_B}.
\]
In particular, $G \isom \frac{A*B}{\normal{r}}$.
\end{corollary}

\begin{proof}
Since $A$ and $B$ embed in $G$, if $r$ lies in the submodule generated by the images of $F_A\cap N$ and $F_B\cap N$, then $G\isom A*B$. Thus, $N = \normal{W_A, W_B}$. Otherwise, since the only units in $\Z$ are $\pm1$, the result follows from \cref{lifting_relations}.
\end{proof}

We may also rephrase \cref{Zlifting} as a statement about $R$-perfect subgroups of one-relator products, upgrading \cref{perfect_subgroup}.

\begin{corollary}
Let $R$ be a domain and let $G = \frac{A*B}{\normal{w}}$ be a one-relator product of locally indicable groups. If $P\triangleleft A*B$ is subgroup containing $w$ such that $P/\normal{w}$ is $R$-perfect and $G/P$ is locally indicable, then there are $R$-perfect subgroups $P_A\leqslant A$, $P_B\leqslant B$ and an element $r\in A*B$ such that $P = \normal{P_A, P_B, r}$.
\end{corollary}

\begin{proof}
Let $F_A/\normal{W_A}$ and $F_B/\normal{W_B}$ be presentations for $A' = A/P_A$ and $B' = B/P_B$, where $P_A = P\cap A$ and $P_B = P\cap B$. By \cref{generating_lemma}, $W_A, W_B, w$ generate the $R$-relation module on $F_A*F_B$ for $H=G/P$. Since $H$ is locally indicable, so are $A', B'$. By the Burns--Hale Theorem \cite{BH72}, $H$ is right orderable. By \cref{lifting_relations}, there is a (cyclically reduced) word $w'\in A'*B'$ such that $H \isom \frac{A'*B'}{\normal{w'}}$. Choosing any preimage $r$ of $w'$ in $A*B$, we see that $P = \normal{P_A, P_B, r}$. By \cref{inclusion}, each proper non-empty subword of $r$ is non-trivial in $H$. Applying \cref{homology_injection}, we see that $P_A$ and $P_B$ are $R$-perfect.
\end{proof}

The first part of the following corollary follows directly from \cref{weak_R-admissibility}, the two further points follow from \cref{lifting_relations}.

\begin{corollary}[No $R$-relation gap]
\label{R-cyclic_module}
Let $R$ be a domain, $F$ a free group, $1\neq N\triangleleft F$ a normal subgroup such that the $R$-relation module $N_R$ for $G = F/N$ is generated as an $RG$-module by a single element. If $G$ is right orderable, then there is some $r\in N$ such that $N = \normal{r}$ and:
\begin{enumerate}
\item If $q\in N_R$ is a generator, there is some unit $u\in R$ and some $f\in F$ such that 
\[
q = u\otimes r^f[N, N].
\]
\item If $w\in N$ such that $q = 1\otimes w[N,N]$, then there is an integer $k\neq 0$ such that $1\otimes k\in R\otimes_{\Z}\Z$ is a unit and some $f\in F$ such that 
\[
q = 1\otimes (r^f)^k[N, N].
\]
\end{enumerate}
In particular, $G = F/\normal{r}$ is a one-relator group.
\end{corollary}

\begin{corollary}[Relation lifting]
\label{cyclic_module}
Let $F$ be a free group, $N\triangleleft F$ a normal subgroup such that the relation module $N_{\Z} = N/[N, N]$ for $G = F/N$ is generated as a $\Z G$-module by a single element $r$. If $G$ is right orderable, then there is some $w\in N$ such that:
\[
w[N, N] = r, \quad N = \normal{r}.
\]
\end{corollary}

\subsection{An example and a conjecture}

\cref{R-cyclic_module} solves the relation gap conjecture for right orderable groups with cyclic relation module over an arbitrary domain. However, \cref{R-cyclic_module} cannot be improved to all groups with cyclic relation module over $\Q$ as the following example with torsion demonstrates.

\begin{example}
\label{example}
We claim that the following two-generator group is not one-relator but has cyclic relation module over $\Q$:
\[
Q_n \isom \langle a, b \mid [tat^{-1}, a] = a^n = 1\rangle
\]
This group appeared in \cite{BT07} where it is shown that, for appropriate choices of $p, q$, the relation module of $Q_p*Q_q$ on the obvious generators is generated by three elements. It is not one-relator as it contains a copy of $\Z/n\Z\times \Z/n\Z$, but torsion subgroups of one-relator groups are cyclic. Denote by $N = \normal{[tat^{-1}, a], a^n}\triangleleft F(a, b)$. We claim that 
\[
\Q\otimes_{\Z}N/[N, N] = \Q Q_n\cdot (1\otimes[tat^{-1}, a]a^{-n}[N, N]).
\]
It will suffice to show that $H_1(N/\normal{[tat^{-1}, a]a^{-n}}, \Q) = 0$ by \cref{generating_lemma}. In other words, that the normal closure $\normal{a_n}_H$ of $a^n$ in $H = F(a, b)/\normal{[tat^{-1}, a]a^{-n}}$ is $\Q$-perfect. But 
\[
[ta^nt^{-1}, a^n] =_H a^{n((n+1)^n - 1)}
\]
by \cite[Lemma 3.2]{BT07} and so the element $a^n[\normal{a^n}_H, \normal{a^n}_H]$ has finite order in $\normal{a^n}/[\normal{a^n}_H, \normal{a^n}_H]$. In particular, this implies that $H_1(\normal{a^n}_H, \Z)$ is torsion and so $H_1(\normal{a^n}_H, \Q) = 0$.
\end{example}

\cref{example} still leaves open the possibility that torsion-free groups with cyclic relation module over $\Q$ are one-relator groups.

\begin{conjecture}
Let $F$ be a free group, $N\triangleleft F$ a normal subgroup and suppose that $G = F/N$ is torsion-free. If $N_{\Q}
 = \Q\otimes_{\Z}N/[N, N]$ is cyclic as a $\Q G$-module, then $N$ is normally generated by a single element.
\end{conjecture}

\bibliographystyle{amsalpha}
\bibliography{bibliography}

\end{document}